\newif\ifappend
\appendtrue
\newif\ifarxiv
\arxivtrue

\documentclass[pmlr,twocolumn,a4paper,10pt]{jmlr}
\usepackage{isipta2023}
\usepackage[british]{babel}
\usepackage{booktabs} 
\usepackage{tikz} 
\usepackage{enumitem}

\newcommand{\noarg}{\bullet}

\newcommand{\deDe}{De}

\newcommand{\reals}{\mathbb{R}}
\newcommand{\extreals}{\overline{\mathbb{R}}}
\newcommand{\nnegreals}{\mathbb{R}_{\geq 0}}
\newcommand{\posreals}{\mathbb{R}_{>0}}
\newcommand{\nats}{\mathbb{N}}
\newcommand{\nnegints}{\mathbb{Z}_{\geq 0}}

\DeclarePairedDelimiter{\pr}{(}{)}
\DeclarePairedDelimiter{\st}{\{}{\}}

\DeclarePairedDelimiter{\br}{[}{]}
\DeclarePairedDelimiter{\abs}{\vert}{\vert}

\DeclarePairedDelimiterX{\ooi}[2]{]}{[}{#1, #2}
\DeclarePairedDelimiterX{\oci}[2]{]}{]}{#1, #2}
\DeclarePairedDelimiterX{\coi}[2]{[}{[}{#1, #2}
\DeclarePairedDelimiterX{\cci}[2]{[}{]}{#1, #2}

\DeclarePairedDelimiter{\norm}{\Vert}{\Vert}
\DeclarePairedDelimiter{\opnorm}{\Vert}{\Vert_{\mathrm{op}}}
\DeclarePairedDelimiter{\zopnorm}{\Vert}{\Vert_{\mathrm{op}}^{0}}

\newcommand{\indica}[1]{\mathbb{I}_{#1}}

\newcommand{\bfns}{\mathcal{L}}  
\newcommand{\gendomain}{\mathcal{D}}
\newcommand{\genfnal}{\Phi}
\newcommand{\altgenfnal}{\Psi}
\newcommand{\genmeas}{\mathcal{M}}
\newcommand{\genbbmeas}{\mathcal{M}_{\mathrm{b}}}
\newcommand{\genbameas}{\mathcal{M}^{\mathrm{b}}}

\newcommand{\eye}{\mathrm{I}}
\newcommand{\genop}{\mathrm{S}}

\newcommand{\genops}{\mathbb{O}}
\newcommand{\bops}{\mathbb{O}_{\mathrm{b}}}

\newcommand{\stsp}{\mathscr{X}}
\newcommand{\bfnsstsp}{\mathscr{L}}  
\newcommand{\utranop}{\overline{\mathrm{T}}}
\newcommand{\urateop}{\overline{\mathrm{Q}}}
\newcommand{\nisiogr}{\overline{\mathrm{N}}}
\newcommand{\nisiogen}{\overline{\mathrm{R}}}
\newcommand{\poissgen}{\overline{\mathrm{L}}}
\newcommand{\linpoissgen}[1]{\mathrm{L}_{#1}}
\newcommand{\poissgr}{\overline{\mathrm{M}}}
\newcommand{\tranop}{\mathrm{T}}
\newcommand{\rateop}{\mathrm{Q}}

\newcommand{\poissint}{\Lambda}
\newcommand{\llambda}{\underline{\lambda}}
\newcommand{\ulambda}{\overline{\lambda}}

\newcommand{\tset}{\nnegreals}
\newcommand{\ftsets}{\mathscr{U}}
\newcommand{\ftset}{U}
\newcommand{\ftsetalt}{V}

\newcommand{\pth}{\omega}
\newcommand{\cadpths}{\Omega}

\newcommand{\fdomain}{\mathscr{D}}
\newcommand{\bbmeas}{\mathscr{M}_{\mathrm{b}}}
\newcommand{\bameas}{\mathscr{M}^{\mathrm{b}}}
\newcommand{\meas}{\mathscr{M}}

\newcommand{\uprev}{\overline{E}}

\newcommand{\prev}{E}
\newcommand{\cprev}{\mathcal{E}}
\newcommand{\cprevext}{\hat{\mathcal{E}}}
\newcommand{\prevs}{\mathbb{E}}
\newcommand{\domprevs}{\mathbb{E}}
\newcommand{\prob}{P}
\newcommand{\altprob}{R}

\newcommand{\cyl}{F}
\newcommand{\cylevts}{\mathscr{F}}

\newcommand{\genset}{\mathcal{Y}}

\ifarxiv
\jmlrworkshop{}
\else
\jmlrworkshop{ISIPTA 2023}
\fi

\title
{
Sublinear Expectations for Countable-State Uncertain Processes
}

\author{
  \Name{Alexander Erreygers}\Email{alexander.erreygers@ugent.be}\\
  \addr Foundations Lab, Ghent University, Belgium
}

\begin{document}




\maketitle

\begin{abstract}
  Sublinear expectations for uncertain processes have received a lot of attention recently, particularly methods to extend a downward-continuous sublinear expectation on the bounded finitary functions to one on the non-finitary functions.
  In most of the approaches the domain of the extension is not very rich because it is limited to bounded measurable functions on the set of all paths.
  This contribution alleviates this problem in the countable-state case by extending, under a mild condition, to the extended real measurable functions on the set of càdlàg paths, and investigates when a sublinear Markov semigroup induces a sublinear expectation that satisfies this mild condition.
\end{abstract}

\section{Introduction}
This contribution is part of the recent push to generalise the theory of (continuous-time) Markov processes from the measure-theoretic framework \citep{1960Chung-Markov,1986EthierKurtz-Markov,1994Rogers-Diffusions,1997Fristedt-Modern} to two closely-related frameworks: those of nonlinear expectations and imprecise probabilities.
These two frameworks have a common aim: to model uncertainty in a (more) robust manner.

The framework of nonlinear---sublinear or even convex---expectations was initially put forward by \citet{2005Peng}, and it dealt with Markov processes from its conception.
Since then the following (and quite possibly more) Markov processes have been generalised to this framework: those with finite state space \cite{2020Nendel}, countable state space \cite{2021Nendel} and \(\reals^d\) as state space \cite{2005Peng}, Lévy processes with \(\reals^d\) as state space \cite{2021HuPeng,2017Neufeld,2020Denk} and Feller processes with a Polish state space \cite{2021NendelRockner}.
Most authors only consider bounded functions on the set of \emph{all} paths that are measurable with respect to the product \(\sigma\)-algebra, which is a domain that is not very rich.
\Citet{2017Neufeld} are a notable exception to this; while they do assume càdlàg paths, they never go beyond functions of the state in a single time point though.

The theory of imprecise probabilities, which is actually a collection of frameworks including those of coherent lower/upper probabilities and sets of (coherent conditional) probabilities, was popularised by \citet{1991Walley}.
In this framework much has been done for Markov processes with finite state space \cite{2015Skulj,2017DeBock,2017KrakDeBockSiebes,2022Erreygers-IJAR}, but to the best of my knowledge, the only work regarding a non-finite state space is that on the Poisson process \cite{2019ErreygersDeBock}.

This contribution builds on and continues the aforementioned work by investigating sublinear expectations for countable-state uncertain processes in general and Markov processes in particular on a sufficiently rich domain of functions on the set of càdlàg paths.
On the one hand, it builds and extends some of the aforementioned work in the framework of nonlinear expectations: \sectionref{sec:nonlinear expectations} rids the extension results in \citep[Section~3]{2018DenkKupperNendel} from the requirement that all functions in the domain are bounded, \sectionref{sec:extension for stoch proc} establishes a more useful version of the robust Daniell--Kolmogorov Extension Theorem \citep[Section~4]{2018DenkKupperNendel} in the specific case of a countable state space by ensuring càdlàg paths, and \sectionref{sec:operators semigroups and beyond} deals with semigroups following \citet[Section~5]{2018DenkKupperNendel} and \citet{2020Nendel,2021Nendel}.
On the other hand, \sectionref{sec:extension for stoch proc,sec:operators semigroups and beyond} essentially extend the approach in \citep{2022Erreygers-IJAR} from the finite-state case to the countable-state case and the approach in~\citep{2019ErreygersDeBock} from the Poisson process to general Markov processes.

\section{Nonlinear Expectations and Their Extensions}
\label{sec:nonlinear expectations}

A (linear) expectation is a linear, normed and isotone (real) functional on a suitable domain.
Examples include the Lebesgue integral with respect to a probability measure and (linear) previsions in the style of \citet{1974DeFinetti-Theory}---see also \citep{2000Whittle-Probability}.
A nonlinear expectation, then, generalises this notion by relaxing the requirement of linearity.
Two examples of such classes of functionals are the convex expectations that appear in (robust) mathematical finance \citep{2005Peng} and the coherent upper---or sublinear---expectations that are at the core of the theory of imprecise probabilities \citep{2014TroffaesDeCooman-Lower,1991Walley}; see also \citep{2003Pelessoni} and references therein.
In \sectionref{ssec:nonlinear expectations} we will formally introduce these well-known notions for domains that may include unbounded and even extended real functions, while \sectionref{ssec:extending nonlinear expectations} deals with extending these nonlinear expectations in such a way that their properties are (partially) preserved.
First, however, we introduce some necessary terminology and notation.

Consider some (non-empty) sets~\(\genset, \mathcal{Z}\).
We denote the set of all maps from~\(\genset\) to~\(\mathcal{Z}\) by~\(\mathcal{Z}^{\genset}\).
For any real-valued function~\(f\) on~\(\genset\), we let
\begin{equation*}
  \norm{f}
  \coloneqq \sup\abs{f}
  = \sup\st{\abs{f\pr{y}}\colon y\in \genset},
\end{equation*}
and we say that \(f\) is \emph{bounded} whenever~\(\norm{f}<+\infty\).
We collect the bounded real-valued functions on~\(\genset\) in~\(\bfns\pr{\genset}\subseteq\reals^{\genset}\).
The bounded real-valued functions on~\(\genset\) include the indicator functions: for any subset~\(Y\) of~\(\genset\), the corresponding \emph{indicator}~\(\indica{Y}\in\bfns\pr{\genset}\) maps~\(y\in\genset\) to~\(1\) if \(y\in Y\) and to~\(0\) otherwise; for any \(y\in\genset\), we shorten~\(\indica{\st{y}}\) to~\(\indica{y}\).

Fix some non-empty subset~\(\gendomain\) of~\(\extreals{}^{\genset}\).
An (extended real) \emph{functional}~\(\genfnal\) (on~\(\gendomain\)) is an extended real map on~\(\gendomain\).
We call a functional~\(\genfnal\) \emph{positively homogeneous} if \(\genfnal\pr{\mu f}=\mu \genfnal\pr{f}\) for all \(f\in\gendomain\) and \(\mu\in\nnegreals\) such that \(\mu f\in\gendomain\), and \emph{homogeneous} if the same equality holds for all \(f\in\gendomain\) and \(\mu\in\reals\) such that \(\mu f\in\gendomain\).
Furthermore, we call a functional~\(\genfnal\) \emph{convex} if \(\genfnal\pr{\lambda f+\pr{1-\lambda}g}\leq\lambda\genfnal\pr{f}+\pr{1-\lambda}\genfnal\pr{g}\) for all \(f,g\in\gendomain\) and \(\lambda\in\cci{0}{1}\) such that \(\lambda f+\pr{1-\lambda}g\) is meaningful\footnote{
  We follow the standard conventions---see for example \citep[Appendix~D]{2014TroffaesDeCooman-Lower}, \citep[Chapter~4, around Definition~8]{1997Fristedt-Modern} or \citep[Table~8.1]{2017Schilling-Measures}---so the sum or difference of two extended reals is meaningful if it does not lead to \(\pr{+\infty}-\pr{+\infty}\) or \(\pr{+\infty}+\pr{-\infty}\).
} and in~\(\gendomain\) and \(\lambda\genfnal\pr{f}+\pr{1-\lambda}\genfnal\pr{g}\) is meaningful, \emph{subadditive} if \(\genfnal\pr{f+g}\leq\genfnal\pr{f}+\genfnal\pr{g}\) for all \(f,g\in\gendomain\) such that \(f+g\) is meaningful and in~\(\gendomain\) and \(\genfnal\pr{f}+\genfnal\pr{g}\) is meaningful, and \emph{additive} if this inequality always holds with equality whenever it is meaningful.
Finally, \(\genfnal\) is said to be \emph{isotone} (sometimes also `order preserving') if \(\genfnal\pr{f}\leq\genfnal\pr{g}\) for all \(f,g\in\gendomain\) such that \(f\leq g\).

We are especially interested in the behaviour of (isotone) functionals with respect to monotone sequences.
A sequence~\(\pr{f_n}_{n\in\nats}\) of extended real functions on~\(\genset\) is said to be \emph{monotone} if it is either increasing or decreasing, where the sequence is \emph{increasing} if \(f_n\leq f_{n+1}\) for all \(n\in\nats\) and \emph{decreasing} if \(f_n\geq f_{n+1}\) for all \(n\in\nats\).
Any monotone sequence~\(\pr{f_n}_{n\in\nats}\) converges pointwise to a limit \(\lim_{n\to+\infty} f_n\in\extreals{}^{\genset}\); \(\pr{f_n}_{n\in\nats}\nearrow f\) denotes an increasing sequence with pointwise limit~\(f\), while \(\pr{f_n}_{n\in\nats}\searrow f\) denotes a sequence that decreases to~\(f\).

Now consider a functional~\(\genfnal\) on~\(\gendomain\).
If \(\genfnal\) is isotone, then for any monotone sequence~\(\pr{f_n}_{n\in\nats}\in\gendomain^{\nats}\), the derived sequence~\(\pr{\genfnal\pr{f_n}}_{n\in\nats}\) is also monotone, and therefore converges to a limit; if furthermore the pointwise limit~\(f\) of~\(\pr{f_n}_{n\in\nats}\) belongs to the domain~\(\gendomain\), then
\begin{equation*}
  \lim_{n\to+\infty} \genfnal\pr{f_n}
  = \inf\st[\big]{\genfnal\pr{f_n}\colon n\in\nats}
  \geq \genfnal\pr{f}.
\end{equation*}
We call the (not necessarily isotone) functional~\(\genfnal\) \emph{downward continuous\footnote{
  Often also `continuous from above'.
} on~\(\mathcal{C}\subseteq\gendomain\)} if for all \(\mathcal{C}^{\nats}\ni\pr{f_n}_{n\in\nats}\searrow f\in\mathcal{C}\), \(\lim_{n\to+\infty} \genfnal\pr{f_n}=\genfnal\pr{f}\), and simply \emph{downward continuous} if it is downward continuous on~\(\gendomain\).
Similarly, we call \(\genfnal\) \emph{upward continuous\footnote{
  Often also `continuous from below'.
} on~\(\mathcal{C}\subseteq\gendomain\)} if for all \(\mathcal{C}^{\nats}\ni\pr{f_n}_{n\in\nats}\nearrow f\in\mathcal{C}\), \(\lim_{n\to+\infty} \genfnal\pr{f_n}=\genfnal\pr{f}\), and simply \emph{upward continuous} if it is upward continuous on~\(\gendomain\).

\subsection{Nonlinear Expectations}
\label{ssec:nonlinear expectations}
A \emph{nonlinear expectation} is a functional~\(\genfnal\) whose domain~\(\gendomain\)---a subset of~\(\extreals{}^{\genset}\)---includes all constant real functions on~\(\genset\) and which is isotone and \emph{constant preserving} (or `normalised'), meaning that \(\genfnal\pr{\mu}=\mu\)\footnote{
  We identify any \(\mu\in\reals\) with the corresponding constant function~`\(\mu\)' on~\(\genset\) whose range is~\(\st{\mu}\).
} for all \(\mu\in\reals\).
Due to isotonicity, any nonlinear expectation~\(\genfnal\) dominates the infimum and is dominated by the supremum; hence, a nonlinear expectation~\(\genfnal\) with domain~\(\gendomain\subseteq\bfns\pr{\genset}\) is a real functional.

This contribution exclusively deals with nonlinear expectations that are convex or even sublinear.
Obviously, a \emph{convex expectation} is a convex nonlinear expectation, while a \emph{sublinear expectation} is a sublinear---so subadditive and positively homogeneous---nonlinear expectation; clearly, any sublinear expectation is also a convex expectation.
An \emph{upper expectation}\footnote{
  Many authors prefer the term `coherent upper prevision/expectation'.
} is a sublinear expectation whose domain~\(\gendomain\) is a linear space of real functions \cite[Theorems~4.15 and 13.34]{2014TroffaesDeCooman-Lower} and a \emph{linear expectation} is an upper expectation that is homogeneous and additive \cite[Theorem~13.36]{2014TroffaesDeCooman-Lower}.

Instead of a convex or sublinear expectation~\(\genfnal\) on~\(\gendomain\), it may be more convenient or appropriate to investigate the corresponding conjugate functional~\(\altgenfnal\) on~\(-\gendomain\coloneqq\st{-f\colon f\in\gendomain}\), defined by the conjugacy relation
\begin{equation*}
  \altgenfnal\pr{f}
  \coloneqq -\genfnal\pr{-f}
  \quad\text{for all } f\in-\gendomain.
\end{equation*}
Clearly, the conjugate functional of a convex expectation is a nonlinear expectation that is concave instead of convex, while that of a sublinear expectation is a nonlinear expectation that is superlinear instead of sublinear.
The conjugate of an upper expectation is what we call a \emph{lower expectation};\footnote{
  The term `coherent lower prevision' is more prevalent.
} for many reference works in the field of imprecise probabilities, including \citep{2014TroffaesDeCooman-Lower,1991Walley}, lower expectations are the primal object under study.

In the remainder of this section, we fix some linear space~\(\gendomain\subseteq\bfns\pr{\genset}\) that contains all constant real functions.
Let \(\prevs_{\gendomain}\) be the set of all linear expectations on \(\gendomain\).
With any convex expectation~\(\cprev\) on~\(\gendomain\), we associate the conjugate function
\begin{equation*}
  \cprev^*
  \colon\prevs_{\gendomain}\to\reals\cup\st{+\infty}
  \colon \prev\mapsto
  \sup\st{\prev\pr{f}-\cprev\pr{f}\colon f\in\gendomain}.
\end{equation*}
This conjugate function plays an important role in the following representation result, taken from \citep[Lemma~2.4]{2018DenkKupperNendel}; for upper expectations, it is known as the \emph{Upper Envelope Theorem} \cite[Theorem 4.38]{2014TroffaesDeCooman-Lower}.
\begin{lemma}
\label{lem:representation}
  For any convex expectation~\(\cprev\) on~\(\gendomain\) and \(f\in\gendomain\),
  \begin{equation*}
    \cprev\pr{f}
    = \max\st[\big]{\prev\pr{f}-\cprev^*\pr{f}\colon \prev\in\prevs_{\gendomain}, \cprev^*\pr{\prev}<+\infty}.
  \end{equation*}
  Furthermore, for any upper expectation~\(\uprev\) on~\(\gendomain\) and linear expectation~\(\prev\in\prevs_{\gendomain}\),
  \begin{align*}
    \uprev{}^*\pr{\prev}<+\infty
    \Leftrightarrow
    \uprev{}^*\pr{\prev}=0
    \Leftrightarrow \pr{\forall g\in\gendomain}~\prev\pr{g}\leq\uprev\pr{g},
  \end{align*}
  so for all \(f\in\gendomain\),
  \begin{align*}
    \cprev\pr{f}
    = \max\st[\big]{\prev\pr{f}\colon \prev\in\prevs_{\gendomain}, \pr{\forall g\in\gendomain}~\prev\pr{g}\leq\uprev\pr{g}}.
  \end{align*}
\end{lemma}

This contribution is mainly concerned with convex expectations~\(\cprev\) that are continuous with respect to monotone sequences.
The following result, taken from \citep[Lemma~3.2]{2018DenkKupperNendel}, links the downward continuity of a convex expectation~\(\cprev\) to the downward continuity of the linear expectations~\(\prev\in\prevs_{\gendomain}\) with \(\cprev^*\pr{\prev}<+\infty\); \citet[Proposition 5.1.2]{2017MirandaZaffalon} give a similar result for the particular case of upper expectations.
\begin{lemma}
\label{lem:downward for convex}
  A convex expectation~\(\cprev\) on~\(\gendomain\) is downward continuous if and only if every linear expectation~\(\prev\) on~\(\gendomain\) with \(\cprev^*\pr{\prev}<+\infty\) is downward continuous.
\end{lemma}

\lemmaref{lem:downward for convex} is crucial in the proof of Theorem~3.10 in~\citep{2018DenkKupperNendel}, and consequently also in our extension of it.
We will see why in due course.

\subsection{Extending Nonlinear Expectations}
\label{ssec:extending nonlinear expectations}
For the remainder of this section, we assume that the linear subspace~\(\gendomain\) of~\(\bfns\pr{\genset}\) that includes all constant functions is also closed under pointwise minima (and then maxima), meaning that for all \(f,g\in\gendomain\), \(f\wedge g\in\gendomain\) (and then \(f\vee g=-\pr{\pr{-f}\wedge\pr{-g}}\in\gendomain\)); this makes \(\gendomain\) a linear lattice \cite[Definition~1.1]{2014TroffaesDeCooman-Lower}.
We let \(\sigma\pr{\gendomain}\) denote the smallest \(\sigma\)-algebra~\(\Sigma\) in~\(\genset\) such that every \(f\in\gendomain\) is \(\Sigma/\mathcal{B}\pr{\reals}\)-measurable---see, for example, \cite[Definition~I.5 b)]{1978DellacherieMeyer-ProbabilitiesA}.
Due to our assumptions, the Daniell--Stone Theorem \citep[Theorem~III.35]{1978DellacherieMeyer-ProbabilitiesA} establishes that if a linear expectation~\(\prev\) on~\(\gendomain\) is downward continuous, then there is a unique probability measure~\(\prob_{\prev}\) on~\(\sigma\pr{\gendomain}\) such that
\begin{equation*}
  \prev\pr{f}
  = \int f\mathrm{d}\prob_{\prev}
  \quad\text{for all } f\in\gendomain.
\end{equation*}
This means that we can extend \(\prev\) to the domain of the Lebesgue integral with respect to~\(\prob_{\prev}\); we adhere to the definition in \citep[Chapter~8, Definition~2]{1997Fristedt-Modern}, so the resulting extension will be extended real valued.
In general, the domain of the Lebesgue integral depends on~\(\prob_{\prev}\); however, it is easy to verify \citep[see][Lemma 24]{2022Erreygers-IJAR} that the domain definitely includes those extended real functions~\(f\) on~\(\genset\) that are \(\sigma\pr{\gendomain}/\mathcal{B}\pr{\extreals}\)-measurable\footnote{
  \(\mathcal{B}\pr{\extreals}\) denotes the Borel \(\sigma\)-algebra generated by the `usual' topology on~\(\extreals\); for more details, see \cite[Chapter~2 and Appendix~C.2]{1997Fristedt-Modern} or \cite[Chapter~8]{2017Schilling-Measures}.
} and either \emph{bounded below}, meaning that \(\inf f>-\infty\), or \emph{bounded above}, meaning that \(\sup f<+\infty\).
In the remainder, we restrict ourselves to such functions; this perhaps somewhat unconventional restriction of the domain is mathematically convenient, and I am yet to encounter an example where the functions of interest do not belong to it.

We collect all extended real functions on~\(\genset\) that are \(\sigma\pr{\gendomain}/\mathcal{B}\pr{\extreals}\)-measurable and bounded below in~\(\genbbmeas\pr{\gendomain}\) and those that are bounded above in~\(\genbameas\pr{\gendomain}\), and we let \(\genmeas\pr{\gendomain}\coloneqq\genbbmeas\pr{\gendomain}\cup\genbameas\pr{\gendomain}\).
Note that \(\genmeas\pr{\gendomain}\) includes
\begin{equation*}
  \gendomain_{\delta, \mathrm{b}}
  \coloneqq\st[\big]{f\in\bfns\pr{\genset}\colon \pr[\big]{\exists\pr{f_n}_{n\in\nats}\in\gendomain^{\nats}}~\pr{f_n}_{n\in\nats}\searrow f}.
\end{equation*}

The following extension result generalises Theorem~3.10 of~\citet{2018DenkKupperNendel} by extending the domain from bounded measurable functions to extended-real measurable functions that are either bounded below or bounded above, so from \(\genmeas\pr{\gendomain}\cap\bfns\pr{\genset}\) to~\(\genmeas\pr{\gendomain}\).
\ifarxiv
The proof remains largely the same, which is why we have relegated it to \appendixref{asec:proofs}.
\fi
\begin{theorem}
\label{the:convex extension}
  Consider a convex expectation~\(\cprev\) on~\(\gendomain\) that is downward continuous, and let \(\cprevext\colon\genmeas\pr{\gendomain}\to\extreals\) be defined for all \(f\in\genmeas\pr{\gendomain}\) by
  \begin{equation*}
    \cprevext\pr{f}
    \coloneqq \sup\st*{\int f\mathrm{d}\prob_{\prev}-\cprev^*\pr{\prev}\colon \prev\in\prevs_{\gendomain}, \cprev^*\pr{\prev}<+\infty}.
  \end{equation*}
  Then \(\cprevext\) is a convex expectation that extends \(\cprev\), is downward continuous on~\(\gendomain_{\delta, \mathrm{b}}\) and upward continuous on~\(\genbbmeas\pr{\gendomain}\).
  Furthermore, if \(\cprev\) is an upper expectation, then \(\cprevext\) is a sublinear expectation.
\end{theorem}

The extension~\(\cprevext\) in \theoremref{the:convex extension} is the unique one with these properties, at least when we restrict its domain to the bounded below functions.
The following result is an immediate corollary of \citep[Theorem~3.10]{2018DenkKupperNendel},
\ifarxiv
so we have relegated its proof to \appendixref{asec:proofs};
\fi
in this result, we write \(\gendomain^{\nats}\ni\pr{f_n}_{n\in\nats}\searrow \leq f\) to mean any decreasing~\(\pr{f_n}_{n\in\nats}\in\gendomain^{\nats}\) such that~\(\lim_{n\to+\infty}f_n\leq f\).
\begin{corollary}
\label{cor:extension is unique}
  Consider a convex expectation~\(\cprevext\) on~\(\genbbmeas\pr{\gendomain}\) that is downward continuous on~\(\gendomain_{\delta, \mathrm{b}}\) and upward continuous.
  Then for all \(f\in\genbbmeas\pr{\gendomain}\),
  \begin{equation*}
    \cprevext\pr{f}
    = \sup\st*{\lim_{n\to+\infty}\cprevext\pr{f_n}\colon \gendomain^{\nats}\ni\pr{f_n}_{n\in\nats}\searrow \leq f}.
  \end{equation*}
\end{corollary}

\section{Countable-State Uncertain Processes}
\label{sec:extension for stoch proc}

With \theoremref{the:convex extension} under our belts, we can move from the general setting of convex expectations to the more specific setting of sublinear expectations for uncertain processes.
With an uncertain process, I mean a `system' that assumes a state in each point of (continuous) time; a `subject' is uncertain about the state as it evolves over time, and we want to model her uncertainty through a sublinear expectation.

As announced by the title, we will focus on uncertain processes whose state takes values in a countable \emph{state space}~\(\stsp\).
An outcome or realisation of such an uncertain process is a map from \(\tset\) to~\(\stsp\), which we call a \emph{path}.
We make---or our subject makes---the common assumptions that (i) after it transitions to a new state, the system stays in its new state for some time; and (ii) that the number of state transitions in every bounded interval is finite.
This is equivalent to assuming that the outcomes are \emph{càdlàg} paths, meaning that they are continuous from the right and have left-sided limits---see, for example, \citep[Chapter~3, Section~5]{1986EthierKurtz-Markov} or \citep[Section~31.1]{1997Fristedt-Modern}.
Thus, our outcome space is the set of all càdlàg paths, which we denote by~\(\cadpths\subseteq\stsp^{\tset}\).
As will become clear at the end of this section, the assumption of càdlàg paths ensures that the domain of the constructed sublinear expectation is sufficiently rich.

Our starting point will be an upper expectation on the set of bounded functions on~\(\cadpths\) that are \emph{finitary}, meaning that they only depend on the states that a path~\(\pth\in\cadpths\) assumes in a finite set of time points.
To formalise this, it will be convenient to introduce some notation.

For all \(\ftset\subseteq\ftsetalt\subseteq\tset\), we let \(\pi_{\ftset}^{\ftsetalt}\colon\stsp^{\ftsetalt}\to\stsp^{\ftset}\) be the projection that maps any \(\stsp\)-valued function on~\(\ftsetalt\) to its restriction to~\(\ftset\) \citep[see][Section~II.25]{1994Rogers-Diffusions}; if \(\ftsetalt=\tset\), we simply write \(\pi_{\ftset}\) and moreover restrict the domain of \(\pi_{\ftset}\) to~\(\cadpths\subseteq\stsp^{\tset}\).
Given some \(\ftset=\st{t_1, \dots, t_n}\), we will sometimes identify \(x=\pr{x_t}_{t\in\ftset}\in\stsp^{\ftset}\) with the \(n\)-tuple~\(\pr{x_1, \dots, x_n}\coloneqq\pr{x\pr{t_1}, \dots, x\pr{t_n}}\in\stsp^n\) of its values.

With this notation, we can formally define the set of finitary bounded functions on~\(\cadpths\):
\begin{equation*}
  \fdomain
  \coloneqq \st[\big]{f\circ\pi_{\ftset}\colon \ftset\in\ftsets, f\in\bfns\pr{\stsp^{\ftset}}}.
\end{equation*}
It is easy to verify that \(\fdomain\subseteq\bfns\pr{\cadpths}\) is a linear lattice, so the results in \sectionref{ssec:extending nonlinear expectations} apply.

Before we get down to using \theoremref{the:alt sufficient condition for downward continuity}, we should establish that an upper expectation~\(\uprev\) on~\(\fdomain\) is in one-to-one correspondence with a `consistent collection of finite-dimensional upper expectations'; especially in \sectionref{ssec:from semigroup to fidi uexp} further on, it will be more convenient to work with the latter.
This notion, essentially taken from \citep[Definition~4.2]{2018DenkKupperNendel}, is the counterpart of the well-known notion of a `consistent collection of finite-dimensional distributions' \citep[Section~II.29]{1994Rogers-Diffusions} for sublinear expectations.
\begin{definition}
\label{def:fidis}
  A \emph{collection of finite-dimensional upper expectations} is a collection~\(\pr{\uprev_{\ftset}}_{\ftset\in\ftsets}\) such that for all \(\ftset\in\ftsets\), \(\uprev_{\ftset}\) is an upper expectation on~\(\bfns\pr{\stsp^{\ftset}}\).
  Such a collection is \emph{consistent} if for all \(\ftset,\ftsetalt\in\ftsets\) with \(\ftset\subseteq\ftsetalt\),
  \begin{equation*}
    \uprev_{\ftset}\pr{f}
    = \uprev_{\ftsetalt}\pr{f\circ \pi_{\ftset}^{\ftsetalt}}
    \quad\text{for all } f\in\bfns\pr{\stsp^{\ftset}}.
  \end{equation*}
\end{definition}

\begin{proposition}
\label{prop:fidis to upper expectation}
  Consider a consistent family~\(\pr{\uprev_{\ftset}}_{\ftset\in\ftsets}\) of finite-dimensional upper expectations.
  Then there is a unique upper expectation~\(\uprev\) on~\(\fdomain\) such that
  \begin{equation*}
    \uprev\pr{f\circ\pi_{\ftset}}
    = \uprev_{\ftset}\pr{f}
    \quad\text{for all } \ftset\in\ftsets, f\in\bfns\pr{\stsp^{\ftset}}.
  \end{equation*}
  Conversely, if \(\uprev\) is an upper expectation on~\(\fdomain\) and for all \(\ftset\in\ftsets\) we let
  \begin{equation*}
    \uprev_{\ftset}
    \colon \bfns\pr{\stsp^{\ftset}}
    \to \reals
    \colon f\mapsto \uprev_{\ftset}\pr{f}
    \coloneqq \uprev\pr{f\circ\pi_{\ftset}},
  \end{equation*}
  then~\(\pr{\uprev_{\ftset}}_{\ftset\in\ftsets}\) is a consistent collection of finite-dimensional upper expectations.
\end{proposition}
\begin{proof}
  The first part of the statement follows almost immediately from the consistency of the family~\(\pr{\uprev_{\ftset}}_{\ftset\in\ftsets}\)---see for example the proof of Proposition~4.4 in~\cite{2018DenkKupperNendel}.
  The second part of the statement follows immediately from the properties of upper expectations.
\end{proof}

Our aim is to use \theoremref{the:convex extension} to extend the upper expectation~\(\uprev\) on the set~\(\fdomain\) of finitary bounded functions to the set~\(\genmeas\pr{\fdomain}\) of \(\sigma\pr{\fdomain}/\mathcal{B}\pr{\extreals}\)-measurable functions that are bounded below or above.
The reader who is familiar with (measure-theoretic) stochastic processes will probably wonder whether the \(\sigma\)-algebra \(\sigma\pr{\fdomain}\) generated by~\(\fdomain\) is equal to the \(\sigma\)-algebra~\(\sigma\pr{\cylevts}\) that is ordinarily used in that setting.
Let us argue that this is indeed the case.

A subset~\(\cyl\) of~\(\cadpths\) is a \emph{cylinder event} \cite[Definition~II.25.4]{1994Rogers-Diffusions}\footnote{
  In general, this definition assumes \(A\in\mathcal{X}^{\ftset}\), where \(\mathcal{X}\) is a \(\sigma\)-algebra on~\(\stsp\).
  Since \(\stsp\) is countable, \(\mathcal{X}=\wp\pr{\stsp}\) and \(\mathcal{X}^{\ftset}=\wp\pr{\stsp^{\ftset}}\).
} if there are some \(\ftset\in\ftsets\) and \(A\subseteq\stsp^{\ftset}\) such that
\begin{equation*}
  \cyl
  = \pi_{\ftset}^{-1}\pr{A}
  \coloneqq \st[\big]{\pth\in\cadpths\colon \pi_{\ftset}\pr{\pth}\in A}.
\end{equation*}
We collect all cylinder events in~\(\cylevts\).
It is easy to verify that
\begin{equation}
\label{eqn:cylevts from indicators}
  \cylevts
  = \st[\big]{\cyl\subseteq\cadpths\colon \indica{\cyl}\in \fdomain},
\end{equation}
and that \(\cylevts\) is an algebra of subsets of~\(\cadpths\); since \(\fdomain\) is a linear space, this implies that
\begin{equation*}
  \mathrm{span}\pr{\st{\indica{\cyl}\colon\cyl\in\cylevts}}
  \subseteq\fdomain.
\end{equation*}
It is a matter of definition chasing to verify that indeed
\begin{equation}
\label{eqn:sigma cylevts vs sigma fdomain}
  \sigma\pr{\cylevts}
  =\sigma\pr{\fdomain};
\end{equation}
the interested reader can find a proof in
\ifarxiv
  \appendixref{asec:proofs}.
\fi
There even is a third way of getting to~\(\sigma\pr{\fdomain}\): since the countable set~\(\stsp\) equipped with the discrete metric is a Polish space, \(\cadpths\) equipped with the so-called `Skorokhod metric' is a Polish space, and the Borel \(\sigma\)-algebra induced by this metric space is precisely~\(\sigma\pr{\cylevts}=\sigma\pr{\fdomain}\); the interested reader may consult Section~31.1 in~\citep{1997Fristedt-Modern} or Section~5 in~\citep[Chapter~3]{1986EthierKurtz-Markov} for the details.

It remains for us to introduce some notation that simplifies the statement of the main result of this section.
We denote the set of \(\sigma\pr{\fdomain}/\mathcal{B}\pr{\extreals}\)-measurable functions that are bounded below by~\(\bbmeas\coloneqq\genbbmeas\pr{\fdomain}\) and the set of those that are bounded above by~\(\bameas\coloneqq\genbameas\pr{\fdomain}\); to simplify our notation, we also let \(\meas\coloneqq\bbmeas\cup\bameas=\genmeas\pr{\gendomain}\).
For all \(\ftset=\st{t_1, t_2}\in\ftsets\), let \(d_{\ftset}^{\neq}\in\bfns\pr{\stsp^{\ftset}}\) be defined for all \(x\in\stsp^{\ftset}\) by \(d_{\ftset}^{\neq}\pr{x}\coloneqq 1\) if \(x\pr{t_1}\neq x\pr{t_2}\) and \(d_{\ftset}^{\neq}\pr{x}\coloneqq 0\) otherwise; similarly, let \(D_{\ftset}^{\neq}\coloneqq\st{x\in\stsp^{\ftset}\colon x\pr{t_1}\neq x\pr{t_2}}\) and note that \(d_{\ftset}^{\neq}=\indica{D_{\ftset}^{\neq}}\).
\begin{theorem}
\label{the:alt sufficient condition for downward continuity}
  Consider a consistent collection~\(\pr{\uprev_{\ftset}}_{\ftset\in\ftsets}\) of finite-dimensional upper expectations that are downward continuous.
  If for all \(n\in\nats\) there is some \(\lambda_n\in\nnegreals\) such that
  \begin{equation*}
    \limsup_{s\to t}\frac{\uprev_{\st{s,t}}\pr[\big]{d_{\st{s,t}}^{\neq}}}{\abs{s-t}}
    \leq \lambda_n
    \quad\text{for all } t\in\cci{0}{n},
  \end{equation*}
  then the corresponding upper expectation~\(\uprev\) of Proposition~\ref{prop:fidis to upper expectation} is downward continuous.
\end{theorem}
A formal proof for this result can be found in
\ifarxiv
  \appendixref{asec:proofs},
\fi
I will only provide a sketch here.

Our starting point is the sublinear expectation~\(\uprev\) on~\(\fdomain\) of Proposition~\ref{prop:fidis to upper expectation}.
To show that this upper expectation~\(\uprev\) is downward continuous, we recall from \lemmaref{lem:representation,lem:downward for convex} that it suffices to show that every dominated linear expectation~\(\prev\in\prevs_{\fdomain}\)---so with \(\prev\pr{g}\leq\uprev\pr{g}\) for all \(g\in\fdomain\)---is downward continuous.

If we were to use the set of all paths~\(\stsp^{\tset}\) as possibility space instead of the set of càdlàg paths~\(\cadpths\), then a linear expectation~\(\prev\) on the set of `finitary functions'~\(\fdomain\) would be downward continuous if and only if for all \(\ftset\in\ftsets\), the derived set function~\(\prob_{\prev}^{\ftset}\colon \wp\pr{\stsp^{\ftset}}\to\cci{0}{1}\colon A\mapsto\prev\pr[\big]{\indica{\pi_{\ftset}^{-1}\pr{A}}}\) is countably additive---this essentially follows from Theorem~2 in \citep{2022Erreygers}, see also \lemmaref{lem:derived set function} further on.
For any dominated linear expectation~\(\prev\in\prevs_{\fdomain}\), this would be guaranteed by the assumption in \theoremref{the:alt sufficient condition for downward continuity} that \(\uprev_{\ftset}\) is downward continuous for all \(\ftset\in\ftsets\).
Note that this is (almost) exactly the approach that \citet{2018DenkKupperNendel} take to prove their Theorem~4.6.

With the set of càdlàg paths~\(\cadpths\) as possibility space, things are more difficult.
For a linear expectation~\(\prev\) on~\(\fdomain\) to be downward continuous, it is still necessary but no longer sufficient that the derived set function~\(\prob_{\prev}^{\ftset}\colon \wp\pr{\stsp^{\ftset}}\to\cci{0}{1}\colon A\mapsto\prev\pr[\big]{\indica{\pi_{\ftset}^{-1}\pr{A}}}\) is downward continuous.
One could obtain a necessary and sufficient condition by combining Theorem~10 in \citep{2022Erreygers} with \lemmaref{lem:derived set function} further on.
Here, however, we will rely on the sufficient condition in \citep[Proposition~24]{2022Erreygers}; we rephrase this result to the context here with the help of Proposition~1, Eqn.~(5) and Theorem~10 there.
\begin{lemma}
\label{lem:prop24}
  Consider a set function~\(\altprob\colon\cylevts\to\cci{0}{1}\) that is finitely additive with \(\altprob\pr{\cadpths}=1\).
  If
  \begin{enumerate}[label=\upshape(\roman*)]
    \item for all \(\ftset\in\ftsets\), the derived set function
      \begin{equation*}
        \altprob^{\ftset}
        \colon \wp\pr{\stsp^{\ftset}}\to\cci{0}{1}\colon
        A\mapsto
        \altprob\pr[\big]{\pi_{\ftset}^{-1}\pr{A}}
      \end{equation*}
      is countably additive; and
      \item for all \(n\in\nats\), there is some \(\lambda_n\in\nnegreals\) such that
      \begin{equation*}
        \limsup_{s\to t} \frac{\altprob^{\st{t,s}}\pr{D_{\st{t,s}}^{\neq}}}{\abs{s-t}}
        \leq \lambda_n
        \quad\text{for all } t\in\cci{0}{n};
      \end{equation*}
  \end{enumerate}
  then \(\altprob\) is countably additive.
\end{lemma}
We can restrict our attention to the derived set function~\(\prob_{\prev}\) due to the following measure-theoretical lemma;
\ifarxiv
I could not immediately find a reference for it, so I give a formal proof in \appendixref{asec:proofs}.
\fi
\begin{lemma}
\label{lem:derived set function}
  Consider a linear expectation~\(\prev\) on~\(\fdomain\).
  Then the derived non-negative set function
  \begin{equation*}
    \altprob_{\prev}
    \colon \cylevts\to\cci{0}{1}
    \colon \cyl\mapsto \prev\pr{\indica{\cyl}}.
  \end{equation*}
  is finitely additive with \(\altprob_{\prev}\pr{\cadpths}=1\).
  Furthermore, \(\prev\) is downward continuous if and only if \(\altprob_{\prev}\) is countably additive; whenever this is the case, \(\altprob_{\prev}=\prob_{\prev}\vert_{\cylevts}\).
\end{lemma}


I would like to emphasise that while the assumption that the outcomes are càdlàg paths makes it more difficult to prove downward continuity, it is also crucial to ensure that the domain~\(\meas\) is sufficiently rich.
Indeed, it is well-known---see, for example, \citep[Section~3.2]{2022Erreygers} or \citep[Lemma~II.25.9]{1994Rogers-Diffusions}---that with the set~\(\stsp^{\tset}\) of all paths as outcome space, the \(\sigma\)-algebra~\(\sigma\pr{\fdomain}=\sigma\pr{\cylevts}\) generated by the cylinder events only contains events that depend on the state of the system in a countable subset of~\(\tset\).
This then implies that~\(\meas\) essentially only contains functions that depend on the state of the system in a countable subset of~\(\tset\), excluding important functions like stopping times and time averages.

In light of the material in \sectionref{sec:nonlinear expectations}, the obvious question is whether \theoremref{the:alt sufficient condition for downward continuity} also holds if we relax the requirement of sublinearity to convexity, so starting from a `consistent collection~\(\pr{\cprev_{\ftset}}_{\ftset\in\ftsets}\) of finite-dimensional convex expectations that are downward continuous.'
Following the same argument as before, \lemmaref{lem:representation,lem:downward for convex} tell us that every linear expectation~\(\prev\in\prevs_{\fdomain}\) with \(\cprev^*\pr{\prev}<+\infty\)---where the convex expectation~\(\cprev\) on~\(\fdomain\) is the unique one that corresponds to \(\pr{\cprev_{\ftset}}_{\ftset\in\ftsets}\) in the sense of Proposition~\ref{prop:fidis to upper expectation} (generalised from sublinear to convex expectations)---should be downward continuous.
To this end, we note that for any such linear expectation~\(\prev\), it follows from \lemmaref{lem:representation} that for all \(s,t\in\tset\) such that \(s\neq t\),
\begin{equation*}
  \frac{\prev\pr{\pi^{-1}_{\st{s,t}}\circ d^{\neq}_{s,t}}}{\abs{s-t}}
  \leq\frac{\cprev_{\st{s,t}}\pr{d^{\neq}_{s,t}}}{\abs{s-t}} + \frac{\cprev^*\pr{\prev}}{\abs{s-t}},
\end{equation*}
and the second term on the right hand side grows to~\(+\infty\) as \(s\) approaches~\(t\).
In other words, we cannot follow the same strategy as before because the presence of~\(\cprev^*\pr{\prev}\) in this inequality does not allow us to invoke \lemmaref{lem:prop24,lem:derived set function}.
One way to alleviate this issue is to rely on the necessary and sufficient condition of Theorem~10 in \citep{2022Erreygers} instead of the sufficient condition of Proposition~24 in \citep{2022Erreygers} (rephrased here as \lemmaref{lem:prop24}), but this would certainly need a different condition on~\(\pr{\cprev_{\ftset}}_{\ftset\in\ftsets}\) as the one in \theoremref{the:alt sufficient condition for downward continuity}; I leave this as an interesting avenue of future research.

\section{Operators, Semigroups and Beyond}
\label{sec:operators semigroups and beyond}
One way to construct a consistent collection of finite-dimensional upper expectations is by glueing together an `initial upper expectation' and an `upper transition semigroup'.
This approach, which is also followed by \citet{2018DenkKupperNendel} and \citet{2020Nendel,2021Nendel}, is inspired by and a generalisation of the well-known construction method for Markov processes, where an `initial distribution' is glued to a `Markov semigroup' to end up with a consistent collection of finite-dimensional distributions (and then a probability measure)---see, for example, \citep[Section~III.7]{1994Rogers-Diffusions} or \citep[Chapter~31]{1997Fristedt-Modern} for more details.

We will get to this construction method in \sectionref{ssec:from semigroup to fidi uexp} further on, but only after we have discussed general operators in \sectionref{ssec:operators} and sublinear Markov semigroups in \sectionref{ssec:sublinear markov semigroups}.

\subsection{Operators}
\label{ssec:operators}
To simplify our notation, let us write \(\bfnsstsp\coloneqq\bfns\pr{\stsp}\).
An \emph{operator}~\(\genop\) is a map from \(\bfnsstsp\) to \(\bfnsstsp\), and we collect all operators in~\(\genops\).
One important operator is the \emph{identity operator}~\(\eye\), which maps any \(f\in\bfnsstsp\) to itself.
For any operator~\(\genop\in\genops\), it is common \cite[see][Section~3.2]{1976Martin-Nonlinear} to call
\begin{equation*}
  \opnorm{\genop}
  \coloneqq \sup\st*{\frac{\norm{\genop f}}{\norm{f}}\colon f\in\bfnsstsp, f\neq 0}
\end{equation*}
its \emph{operator seminorm}; we call an operator~\(\genop\) \emph{bounded} if \(\opnorm{\genop}<+\infty\), and collect all bounded operators in~\(\bops\).
It is easy to verify that \(\opnorm{\noarg}\) is a seminorm on~\(\bops\).
Furthermore, with a bit more work one can verify that
\begin{equation*}
  \zopnorm{\noarg}
  \colon \bops\to\nnegreals
  \colon \genop\mapsto \norm{\genop 0} + \opnorm{\genop}
\end{equation*}
is a norm, and that the space~\(\pr{\bops, \zopnorm{\noarg}}\) is a Banach space; the proof is essentially the same as Martin's~\citep[Section~III.2]{1976Martin-Nonlinear}, who instead of the operator seminorm~\(\opnorm{\noarg}\) considers the Lipschitz seminorm.
Clearly, the identity operator~\(\eye\) is bounded with \(\zopnorm{\eye}=\opnorm{\eye}=1\).

We extend (almost) all of the terminology regarding properties of functionals to operators in the obvious `componentwise' manner; for example, an operator~\(\genop\in\genops\) is called subadditive if for all \(x\in\stsp\), the component functional
\begin{equation*}
  \br{\genop\noarg}\pr{x}\colon\bfnsstsp\to\reals
  \colon f\mapsto \br{\genop f}\pr{x}
\end{equation*}
is subadditve.
If \(\genop\in\bops\) is positively homogeneous, then \(\genop 0=0\) and
\begin{equation*}
  \zopnorm{\genop}
  = \opnorm{\genop}
  = \sup\st*{\norm{\genop f}\colon f\in\bfnsstsp, \norm{f}=1},
\end{equation*}
which is in accordance with the operator norm for positively homogeneous operators used in \citep[Eqn.~(1)]{2017KrakDeBockSiebes} and \citep[Eqn.~(4)]{2017DeBock} and the standard norm for linear---additive and homogeneous---operators \cite[Section~23.1]{1997Schechter-Handbook}.

There are two classes of (positively homogeneous) operators that play an important role in the remainder, and we will see in \sectionref{ssec:examples} further on that these are closely connected.
The first class is that of the \emph{upper transition operators}: those operators~\(\utranop\) such that for all \(x\in\stsp\), the corresponding component functional \(\br{\utranop\noarg}\pr{x}\) is an upper expectation (so isotone, constant preserving, subadditive and positively homogeneous); this terminology goes back to \citet[Section~8]{2008deCoomanHermans}, but note that \citet[Definition~5.1]{2018DenkKupperNendel} prefer the term `sublinear kernel'.
It follows immediately from the properties of upper expectations \citep[see][Theorem~4.13]{2014TroffaesDeCooman-Lower} that for any upper transition operator~\(\utranop\),
\begin{equation}
\label{eqn:utranop:contraction}
  \norm{\utranop f-\utranop g}
  \leq\norm{f-g}
  \quad\text{for all } f,g\in\bfnsstsp.
\end{equation}
We will be particularly concerned with upper transition operators that are downward continuous; one example is the identity operator~\(\eye\).
Note that this requirement of downward continuity is trivially satisfied whenever the state space~\(\stsp\) is finite: in this case pointwise convergence implies uniform convergence, and every upper expectation is well known to be continuous (so in particular from above) with respect to uniform convergence \citep[Theorem~4.13~(xiii)]{2014TroffaesDeCooman-Lower}.

The second important class are those operators~\(\genop\) that satisfy the \emph{positive maximum principle},\footnote{
  After \citet[Section~1.2]{1965Courrege}, see also \citep[Chapter~4, Section~2]{1986EthierKurtz-Markov} or \citep[Lemma~III.6.8]{1994Rogers-Diffusions}.
} meaning that \(\br{\genop f}\pr{x}\leq0\) for all \(f\in\bfnsstsp\) and \(x\in\stsp\) such that \(\sup f=f\pr{x}\geq0\).
Of special interest is the subclass of \emph{upper rate operators}, which are the sublinear and positively homogeneous operators~\(\urateop\) that map constant functions to~\(0\) and satisfy the positive maximum principle---see \citep[Definition~2.1 and Theorem~2.5]{2020Nendel}, \citep[Definition~5]{2017DeBock} or \citep[Definition~7.2]{2017KrakDeBockSiebes}.
A \emph{rate operator} is an upper rate operator that is linear.

In the case of a finite state space, \citet[Eqn.~(9)]{2017DeBock} shows that any upper rate operator is necessarily bounded.
Moreover, \citet[Eqn.~(38) and Proposition~7.6]{2017KrakDeBockSiebes} show that any bounded set~\(\mathcal{Q}\) of rate operators induces an upper rate operator through its pointwise upper envelope~\(\urateop_{\mathscr{Q}}\in\bops\), which maps \(f\in\bfnsstsp\) to
\begin{equation*}
  \urateop_{\mathscr{Q}} f
  \colon \stsp\mapsto \reals
  \colon x\mapsto \sup\st[\big]{\br{\rateop f}\pr{x}\colon \rateop\in\mathscr{Q}};
\end{equation*}
conversely, any upper rate operator~\(\urateop\) is the upper envelope of the corresponding (bounded) set of dominated rate operators
\begin{equation*}
  \mathscr{Q}_{\urateop}
  \coloneqq \st[\big]{\rateop\in\mathfrak{Q}\colon \pr{\forall f\in\bfnsstsp}~\rateop f\leq \urateop f},
\end{equation*}
where \(\mathfrak{Q}\) is the set of all rate operators.
These results also hold for bounded upper rate operators in the case that \(\stsp\) is not finite, but due to lack of space we will not pursue this matter.

\subsection{Sublinear Markov Semigroups}
\label{ssec:sublinear markov semigroups}
A \emph{semigroup} is a family~\(\pr{\genop_t}_{t\in\tset}\) of operators such that
\begin{enumerate}[label=SG\arabic*., ref=\upshape(SG\arabic*), leftmargin=*]
  \item\label{def:utsg:semigroup} \(\genop_{s+t}=\genop_{s}\genop_{t}\) for all \(s,t\in\posreals\), and
  \item\label{def:utsg:identity} \(\genop_0=\eye\);
\end{enumerate}
such a semigroup~\(\pr{\genop_t}_{t\in\tset}\) is \emph{strongly continuous} if
\begin{equation*}
  \lim_{t\searrow 0} \norm{\genop_tf-\genop_0f}
  = 0
  \quad\text{for all } f\in\bfnsstsp
\end{equation*}
and, whenever \(\genop_t\in\bops\) for all \(t\in\tset\), called \emph{uniformly continuous} if
\begin{equation*}
  \lim_{t\searrow 0} \zopnorm{\genop_t-\genop_0}
  = 0.
\end{equation*}

We will only be concerned with semigroups~\(\pr{\utranop_t}_{t\in\posreals}\) of upper transition operators, which we will briefly call \emph{upper transition semigroups}; in this context, the semigroup property~\ref{def:utsg:semigroup} is often called the `Chapman-Kolmogorov equation.'
A \emph{sublinear Markov semigroup},\footnote{
  This definition generalises the notion of a `Markov semigroup' in \citep[Section~III.3]{1994Rogers-Diffusions} to the sublinear case.
} then, is an upper transition semigroup~\(\pr{\utranop_t}_{t\in\tset}\) such that for all \(t\in\posreals\), \(\utranop_t\) is downward continuous.

Due to Eqn.~\eqref{eqn:utranop:contraction}, an upper transition semigroup~\(\pr{\utranop_t}_{t\in\tset}\) is almost a special case of what \citet[Chapter~3]{1992Miyadera-Nonlinear} and \citet[Chapter~III]{1976Barbu-Nonlinear} call a `semigroup of contractions'; the only property that is missing is strong continuity.
That said, we need a---at first sight---slightly different notion of continuity than strong or uniform continuity: we say that an upper transition semigroup~\(\pr{\utranop_t}_{t\in\tset}\) \emph{has uniformly bounded rate} (after \citet[Eqn.~(2.35)]{1991Anderson-Continuous}) if
\begin{equation}
\label{eqn:uniformly bounded rate}
  \limsup_{t\searrow0} \frac1{t}\sup\st[\big]{\br[\big]{\utranop_t\pr{1-\indica{x}}}\pr{x}\colon x\in\stsp}
  < +\infty;
\end{equation}
due to constant additivity, this implies that
\begin{equation}
\label{eqn:uniformly continuous}
  \lim_{t\searrow0} \sup\st[\big]{1+\br{\utranop_t\pr{-\indica{x}}}\pr{x}\colon x\in\stsp}
  = 0.
\end{equation}
The condition of uniformly bounded rate might seem rather strong, but we are not the first to encounter its usefulness and/or necessity; this will become clear in~\sectionref{ssec:examples} further on, where we give some examples of (sub)linear Markov semigroups that satisfy this condition.

\subsection{From Upper Transition Semigroup to Finite-Dimensional Upper Expectations}
\label{ssec:from semigroup to fidi uexp}
Now that we have properly introduced the necessary concepts, we can get down to glueing an initial upper expectation to a semigroup of upper transition operators in such a way that we end up with a consistent collection of finite-dimensional upper expectations.
\begin{proposition}
\label{prop:from semigroup to fidis}
  Consider an upper expectation~\(\uprev_0\) on~\(\bfnsstsp\) and an upper transition semigroup~\(\pr{\utranop_t}_{t\in\tset}\).
  Then there is a unique consistent collection~\(\pr{\uprev_{\ftset}}_{\ftset\in\ftsets}\) of finite-dimensional upper expectations such that
  \begin{enumerate}[label=\upshape(\roman*), ref=\upshape(\roman*)]
    \item\label{prop:from semigroup to fidis:zero} \(\uprev_{\st{0}}\pr{f}=\uprev_0\pr{f}\) for all \(f\in\bfns\pr{\stsp}\); and
    \item\label{prop:from semigroup to fidis:backwards recursion} for all \(\ftset=\st{s_1, \dots, s_n}\in\ftsets\), \(t\in\tset\) such that \(s_1<\cdots<s_n<t\) and \(f\in\bfns\pr{\stsp^{\ftset\cup\st{t}}}\),
    \begin{equation*}
      \uprev_{\ftset\cup\st{t}}\pr{f}
      = \uprev_{\ftset}\pr{g},
    \end{equation*}
    where \(g\in\bfns\pr{\stsp^{\ftset}}\) maps \(x=\pr{x_s}_{s\in\ftset}\in\stsp^{\ftset}\) to
    \begin{equation*}
      g\pr{x}
      \coloneqq \br[\big]{\utranop_{t-s_n}f\pr{x_{s_1}, \dots, x_{s_n}, \noarg}}\pr{x_{s_n}}.
    \end{equation*}
  \end{enumerate}
\end{proposition}
\begin{proof}
  The argument is the same as in the first part of the proof of Theorem~5.6 in~\citep{2018DenkKupperNendel}.
  The family~\(\pr{\uprev_{\ftset}}_{\ftset\in\ftsets}\) is constructed recursively, and the semigroup property~\ref{def:utsg:semigroup} ensures that the constructed family satisfies the condition for consistency in~\definitionref{def:fidis}.
\end{proof}

The consistent collection of finite-dimensional upper expectations in Proposition~\ref{prop:from semigroup to fidis} corresponds to a unique upper expectation~\(\uprev\) on~\(\fdomain\) due to Proposition~\ref{prop:fidis to upper expectation}.
Now the question naturally arises whether \(\uprev\) is downward continuous, because then we can invoke \theoremref{the:convex extension} to extend \(\uprev\) to~\(\meas\).
With the help of \theoremref{the:alt sufficient condition for downward continuity}, we get the following sufficient condition, which I believe is one of the main results of this contribution.
\begin{theorem}
\label{the:Markov semigroup with unif bounded rate then downward}
  Consider a downward-continuous upper expectation~\(\uprev_0\) on~\(\bfnsstsp\) and a sublinear Markov semigroup~\(\pr{\utranop_t}_{t\in\tset}\).
  Then the unique corresponding upper expectation~\(\uprev\) on~\(\fdomain\) induced by Propositions~\ref{prop:from semigroup to fidis} and \ref{prop:fidis to upper expectation} is downward continuous.
\end{theorem}
\begin{proof}
  It suffices to verify that the corresponding consistent collection~\(\pr{\uprev_{\ftset}}_{\ftset\in\ftsets}\) of finite dimensional distributions satisfies the conditions in Theorem~\ref{the:alt sufficient condition for downward continuity}.

  First, we prove that for all \(\ftset\in\ftsets\), \(\uprev_{\ftset}\) is downward continuous.
  For \(\ftset=\st{0}\), \(\uprev_{\ftset}=\uprev_0\) by Proposition~\ref{prop:fidis to upper expectation}~\ref{prop:from semigroup to fidis:zero}, so \(\uprev_{\ftset}\) is downward continuous because \(\uprev_0\) is downward continuous as per the assumptions in the statement.
  Next, we fix any \(\ftset=\st{t_0,t_1, \dots, t_n}\in\ftsets\) with \(n\geq 1\) and \(0=t_0<t_1<\cdots<t_n\).

  For all \(h\in\bfns\pr{\stsp^{\ftset}}\), it follows from repeated application of Proposition~\ref{prop:fidis to upper expectation}~\ref{prop:from semigroup to fidis:backwards recursion} and a single application of~Proposition~\ref{prop:fidis to upper expectation}~\ref{prop:from semigroup to fidis:zero} that
  \begin{equation}
  \label{eqn:proof of the:MSwUBRtCfA:reduction}
    \uprev_{\ftset}\pr{h}
    = \uprev_0\pr{g_0},
  \end{equation}
  where the sequence \(g_0, \dots, g_n\) is derived recursively from the initial condition \(g_n\coloneqq h\) and, for all \(k\in\st{0, \dots, n-1}\), \(g_k\in\bfns\pr{\stsp^{\st{t_0, \dots, t_k}}}\) is defined recursively for all \(x=\pr{x_0, \dots, x_k}\in\stsp^{\st{t_0, \dots, t_k}}\) by
  \begin{equation}
  \label{eqn:proof of the:MSwUBRtCfA:recursion}
    g_k\pr{x}
    \coloneqq \br[\big]{\utranop_{\pr{t_{k+1}-t_k}} g_{k+1}\pr{x_0, \dots, x_k, \noarg}}\pr{x_k}.
  \end{equation}

  Fix any \(\bfns\pr{\stsp^{\ftset}}^{\nats}\ni\pr{f_\ell}_{\ell\in\nats}\searrow f\in\bfns\pr{\stsp^{\ftset}}\).
  Then by Eqn.~\eqref{eqn:proof of the:MSwUBRtCfA:reduction} for \(h=f\), \(\uprev_{\ftset}\pr{f}=\uprev_0\pr{g_0}\), where \(\pr{g_{0}, \dots, g_n}\) is the sequence as defined in Eqn.~\eqref{eqn:proof of the:MSwUBRtCfA:recursion}; similarly, for all \(\ell\in\nats\), \(\uprev_{\ftset}\pr{f_\ell}=\uprev_0\pr{g_{\ell, 0}}\) where \(\pr{g_{\ell, 0}, \dots, g_{\ell, n}}\) is recursively defined as in Eqn.~\eqref{eqn:proof of the:MSwUBRtCfA:recursion} with initial condition~\(g_{\ell, n}=f_\ell\).
  Then for \(k=n-1\), and then for \(k=n-2\) to~\(k=0\), it follows from the downward continuity of~\(\utranop_{t_{k+1}-t_k}\) that \(\pr{g_{\ell, k}}_{\ell\in\nats}\searrow g_k\).
  From this and the downward continuity of \(\uprev_0\), we infer that
  \begin{equation*}
    \lim_{\ell\to+\infty} \uprev_{\ftset}\pr{f_\ell}
    = \lim_{\ell\to+\infty} \uprev_0\pr{g_{\ell, 0}}
    = \uprev_0\pr{g_0}
    = \uprev_{\ftset}\pr{f},
  \end{equation*}
  as required.
  Finally, for all \(\ftset=\st{t_1, \dots, t_n}\in\ftsets\) with \(0<t_1<\cdots<t_n\), the downward continuity of~\(\uprev_{\st{0}\cup\ftset}\)---which we have just proved---implies that of \(\uprev_{\ftset}\) because \(\pr{\uprev_{\ftsetalt}}_{\ftsetalt\in\ftsets}\) is consistent.

  Second, we set out to prove that for all \(n\in\nats\), there is some \(\lambda_n\in\nnegreals\) such that
  \begin{equation}
  \label{eqn:proof of the ...:to prove second part}
    \limsup_{s\to t}\frac{\uprev_{\st{s,t}}\pr[\big]{d_{\st{s,t}}^{\neq}}}{\abs{s-t}}
    \leq \lambda_n
    \quad\text{for all } t\in\cci{0}{n}.
  \end{equation}
  Since \(\pr{\utranop_t}_{t\in\tset}\) has uniformly bounded rate,
  \begin{equation}
  \label{eqn:proof of the ...:lambda}
    \lambda
    \coloneqq
    \limsup_{\Delta\searrow 0} \sup\st*{\frac{\br[\big]{\utranop_{\Delta}\pr{1-\indica{x}}}\pr{x}}{\Delta}\colon x\in\stsp}
    < +\infty.
  \end{equation}
  For all \(t_1,t_2\in\tset\) such that \(t_1<t_2\) and \(\pr{x_{t_1}, x_{t_2}}\in\stsp^{\st{t_1,t_2}}\),
  \begin{equation*}
    d^{\neq}_{\st{t_1,t_2}}\pr{x_{t_1}, x_{t_2}}=1-\indica{x_{t_1}}\pr{x_{t_2}},
  \end{equation*}
  so it follows from Proposition~\ref{prop:fidis to upper expectation}, and with \(\Delta\coloneqq t_2-t_1\), that
  \begin{equation*}
    \uprev_{\st{t_1,t_2}}\pr[\big]{d^{\neq}_{\st{t_1,t_2}}}
    = \uprev_0\pr[\big]{\utranop_{t_1} \pr[\big]{\stsp\to\reals\colon x\mapsto \br[\big]{\utranop_{\Delta}\pr{1-\indica{x}}}\pr{x}}}.
  \end{equation*}
  Since \(\utranop_{t_1}\) and \(\uprev_0\) are bounded above by the supremum and positively homogeneous, it follows more or less immediately from the preceding equality and Eqn.~\eqref{eqn:proof of the ...:lambda} that
  \begin{equation*}
    \limsup_{s\to t} \frac{\uprev_{\st{s,t}}\pr[\big]{d^{\neq}_{\st{s,t}}}}{\abs{s-t}}
    \leq \lambda
    \quad\text{for all } t\in\tset.
  \end{equation*}
  So for all \(n\in\nats\), Eqn.~\eqref{eqn:proof of the ...:to prove second part} holds with \(\lambda_n=\lambda\).
\end{proof}

By combining \theoremref{the:Markov semigroup with unif bounded rate then downward} with \theoremref{the:convex extension} and \corollaryref{cor:extension is unique}, we get what I believe to be the second main result of this contribution.
In order to highlight the similarity to \citep[Theorem 5.6]{2018DenkKupperNendel}, \citep[Theorem 2.5]{2020Nendel} and \citep[Definition 5.5]{2021Nendel}, for all \(t\in\tset\) we let \(X_t\colon\cadpths\to\stsp\) be the projector that maps~\(\pth\in\cadpths\) to~\(\pth\pr{t}\).
\begin{theorem}
\label{the:everything combined}
  Consider an upper expectation~\(\uprev_0\) that is downward continuous and a sublinear Markov semigroup~\(\pr{\utranop_t}_{t\in\tset}\) with uniformly bounded rate.
  Then there is a sublinear expectation~\(\uprev\) on~\(\meas\) such that
  \begin{enumerate}[label=\upshape(\roman*)]
    \item for all \(f\in\bfns\pr{\stsp}\), \(\uprev\pr{f\pr{X_0}}=\uprev_0\pr{f}\);
    \item for all \(\ftset=\st{s_1, \dots, s_n}\in\tset\) and \(t\in\tset\) such that \(s_1<\cdots<s_n<t\) and all \(f\in\bfns\pr{\stsp^{n+1}}\),
    \begin{equation*}
      \uprev\pr{f\pr{X_{t_1}, \dots, X_{t_n}, X_t}}
      = \uprev\pr{g\pr{X_{t_1}, \dots, X_{t_n}}}
    \end{equation*}
    where \(g\in\bfns\pr{\stsp^{n}}\) maps \(x=\pr{x_1, \dots, x_n}\in\stsp^n\) to
    \begin{equation*}
      g\pr{x}
      \coloneqq \br[\big]{\utranop_{t-t_n} f\pr{x_{s_1}, \dots, x_{s_n}, \noarg}}\pr{x_{s_n}};
    \end{equation*}
    \item \(\uprev\) is downward continuous on~\(\fdomain_{\delta, \mathrm{b}}\); and
    \item \(\uprev\) is upward continuous on~\(\bbmeas\);
  \end{enumerate}
  Moreover, the restriction of~\(\uprev\) to~\(\bbmeas\) is the unique sublinear expectation that has these four properties.
\end{theorem}
\begin{proof}
  Follows from \theoremref{the:Markov semigroup with unif bounded rate then downward}, Propositions~\ref{prop:from semigroup to fidis} and \ref{prop:fidis to upper expectation}, \theoremref{the:convex extension} and \corollaryref{cor:extension is unique}.
\end{proof}

In comparison to Theorem~5.6 in \cite{2018DenkKupperNendel}, \theoremref{the:everything combined} has a more limited scope, since the former involves convex expectations, a Polish space as state space and a `two-parameter semigroup'.
\theoremref{the:everything combined} is more useful though, since the domain in this result is much more rich than the one in \citep[Theorem~5.6]{2018DenkKupperNendel}---which only includes bounded functions on the set of all paths that are measurable with respect to the (for many purposes inadequate) product \(\sigma\)-algebra.
A similar comparison can be made between \theoremref{the:everything combined} on the one hand and Theorem~2.5 in \citep{2020Nendel} and Remark~5.4 \citep{2021Nendel}---which rely on Theorem~5.6 in \citep{2018DenkKupperNendel}--- on the other hand, the difference with before being that all of these results assume a countable state space and start from a one-parameter semigroup.

\subsection{Examples of Sublinear Markov Semigroups}
\label{ssec:examples}
\theoremref{the:everything combined} is only useful if we can actually construct and/or determine a sublinear Markov semigroup.
To conclude this section, we list some existing results which do exactly that.

\paragraph*{The Linear Case}
Consider a (linear) \emph{Markov semigroup}~\(\pr{\tranop_t}_{t\in\tset}\), that is, a sublinear Markov semigroup that consists of linear (transition) operators.
Then it follows from our assumptions---and the Daniell--Stone Theorem---that the matrix representation of~\(\pr{\tranop_t}_{t\in\tset}\) is in one-to-one correspondence with what is known as a transition (matrix) function---sometimes also `transition matrix', see \citep[§~1.1]{1991Anderson-Continuous}, \citep[Example~III.3.6]{1994Rogers-Diffusions}, \citep[Section~23.10]{1957Hille-Functional} and \citep[Part~II, §1]{1960Chung-Markov}.
The semigroup~\(\pr{\tranop_t}_{t\in\tset}\) having uniformly bounded rate not only suffices for the condition in Eqn.~\eqref{eqn:uniformly continuous} but is also necessary---see for example \citep[Section~23.11]{1957Hille-Functional} or \citep[Section~II.19, Theorem~2]{1960Chung-Markov}---and under this condition, the Markov semigroup~\(\pr{\tranop_t}_{t\in\tset}\) is generated by a unique bounded linear operator~\(\rateop\), in the sense that for all \(t\in\posreals\),
\begin{equation}
  \tranop_t
  = e^{t\rateop}
  = \lim_{n\to+\infty} \pr*{\eye+\frac{t}{n}\rateop}^n
  = \sum_{n=0}^{+\infty} \frac{t^n\rateop^n}{n!};
\end{equation}
this unique operator~\(\rateop\) is given by
\begin{equation*}
  \rateop
  = \lim_{t\searrow0} \frac{\tranop_t-\eye}{t}
\end{equation*}
and is a rate operator that is downward continuous.

Conversely, any bounded rate operator~\(\rateop\in\bops\) generates a semigroup~\(\pr{e^{t\rateop}}_{t\in\tset}\) of transition operators that has uniformly bounded rate; even more, this semigroup is not only strongly continuous but also uniformly continuous.
Finally, it can be verified that \(\pr{e^{t\rateop}}_{t\in\tset}\) is a Markov semigroup---or equivalently, that for all \(t\in\posreals\), \(e^{t\rateop}\) is downward continuous---if (and then only if) \(\rateop\) is downward continuous; this is trivially the case whenever \(\stsp\) is finite.
For more details regarding linear (uniformly continuous) semigroups and exponentials of bounded linear operators, we refer the reader to \citep{1957Hille-Functional,2000EngelNagel-Semigroups} and references therein.

\paragraph{The Finite-State Case}
The aforementioned results in the linear case translate to the sublinear case, at least when the state space~\(\stsp\) is finite.
In my doctoral dissertation \citep[Theorem~3.75 and Lemma~3.76]{2021Erreygers-Phd}, I show that an upper transition semigroup~\(\pr{\utranop_t}_{t\in\tset}\) has uniformly bounded rate if and only if it satisfies Eqn.~\eqref{eqn:uniformly continuous}, and in that case
\begin{equation*}
  \urateop
  \coloneqq \lim_{t\searrow} \frac{\utranop_t-\eye}{t}
\end{equation*}
is the unique (necessarily bounded) upper rate operator such that
\begin{equation*}
  \utranop_t
  = e^{t\urateop}
  \coloneqq \lim_{n\to+\infty} \pr*{\eye+\frac{t}{n}\urateop}^n
  \quad\text{for all } t\in\tset.
\end{equation*}
Conversely, for any upper rate operator~\(\urateop\), the corresponding family of exponentials~\(\pr{e^{t\urateop}}_{t\in\tset}\) is an upper transition semigroup that has uniformly bounded rate and is uniformly (and therefore strongly) continuous \citep{2017KrakDeBockSiebes,2017DeBock,2015Skulj}.
\Citet[Proposition~4.9]{2020Nendel} show that for any bounded set~\(\mathscr{Q}\subseteq\mathfrak{Q}\) with upper envelope~\(\urateop\), \(\pr{e^{t\urateop}}_{t\in\tset}\) is the pointwise smallest semigroup such that \(e^{t\rateop}f\leq e^{t\urateop} f\) for all \(\rateop\in\mathscr{Q}\) and \(t\in\posreals\).

\paragraph{The Infinite-State Case}
The remaining case that \(\stsp\) is countably infinite has received less attention.
\Citet[Section~5]{2021Nendel} constructs a strongly continuous upper transition semigroup~\(\pr{\nisiogr_t}_{t\in\tset}\) as follows.
The starting point is a set~\(\st{\pr{\tranop^i_t}_{t\in\tset}\colon i\in\mathcal{I}}\) of Markov semigroups.
He then constructs the induced \emph{Nisio semigroup}~\(\pr{\nisiogr_t}_{t\in\tset}\), and shows that it is the pointwise smallest semigroup that dominates each of the Markov semigroups in this set.
In the special case that \(\pr{\tranop^i_t}_{t\in\tset}\) has uniformly bounded rate for all \(i\in\mathcal{I}\), we collect their generators in
\begin{equation*}
  \mathscr{R}
  \coloneqq \st*{\lim_{t\searrow0}\frac{\tranop^i_t-\eye}{t}\colon i\in\mathcal{I}};
\end{equation*}
under the assumption that \(\sup\st{\norm{\rateop}\colon\rateop\in\mathscr{R}}<+\infty\),
\citet[Remark~5.6]{2021Nendel} shows that the Nisio semigroup~\(\pr{\nisiogr_t}_{t\in\tset}\) is strongly continuous, and that
\begin{equation*}
  \lim_{t\searrow0} \frac{\nisiogr_t f-f}{t}
  = \nisiogen f
  \quad\text{for all } f\in\bfnsstsp,
\end{equation*}
where \(\nisiogen\) is the upper envelope of the (bounded) set~\(\mathscr{R}\) of generators of the Markov semigroups.

Unfortunately, it is not immediately clear whether there are conditions on the set of Markov semigroups under which the induced Nisio semigroup~\(\pr{\nisiogr_t}_{t\in\tset}\) is a sublinear Markov semigroup that has uniformly bounded rate.
In order to adhere to the page limit, I've chosen not to pursue this avenue in its full generality here; we will, however, treat one example where this is the case.

Henceforth, we let \(\stsp\coloneqq\nnegints\), and we fix some rate interval~\(\poissint\coloneqq\cci{\llambda}{\ulambda}\subset\nnegreals\).
We define the \emph{sublinear Poisson generator}~\(\poissgen\colon\bfnsstsp\to\bfnsstsp\) for all \(f\in\bfnsstsp\) by
\begin{equation*}
  \poissgen f
  \colon\nnegints\to\reals
  \colon z\mapsto
  \max\st{\lambda\pr{f\pr{z+1}-f\pr{z}}\colon \lambda\in\poissint}.
\end{equation*}
It is easy to verify that \(\poissgen\) is a bounded upper rate operator.
The sublinear Poisson generator~\(\poissgen\) is the pointwise upper envelope of a bounded set of bounded rate operators, for example the sets~\(\st{\linpoissgen{\lambda}\colon \lambda\in\poissint}\) and \(\st{\linpoissgen{\llambda}, \linpoissgen{\ulambda}}\) of Poisson generators, where for all \(\lambda\in\posreals\), the corresponding Poisson generator~\(\linpoissgen{\lambda}\in\bops\) is defined for all \(f\in\bfnsstsp\) by
\begin{equation*}
  \linpoissgen{\lambda} f
  \colon\nnegints\to\reals
  \colon z\mapsto
  \lambda \pr{f\pr{z+1}-f\pr{z}}.
\end{equation*}

\Citet[Theorem~10]{2019ErreygersDeBock} show that for all \(t\in\posreals\),
\begin{equation*}
  \poissgr_t
  \coloneqq \lim_{n\to+\infty} \pr*{\eye+\frac{t}{n}\poissgen}^n
\end{equation*}
exists.
Furthermore, it is shown in~\citep[Theorem~45, Proposition~48 and Lemma~51]{2019ErreygersDeBock-arXiv} that \(\pr{\poissgr_t}_{t\in\tset}\) is a uniformly continuous semigroup of upper transition operators with
\begin{equation}
\label{eqn:poissgr:poissgen is derivative}
  \lim_{t\searrow0} \frac{\poissgr_t-\eye}{t}
  = \poissgen.
\end{equation}
First and foremost, Eqn.~\eqref{eqn:poissgr:poissgen is derivative} implies that \(\pr{\poissgr_t}_{t\in\tset}\) is the Nisio semigroup induced by the sets~\(\st{\pr{M_{\lambda, t}}_{t\in\tset}\colon \lambda\in\Lambda}\) and \(\st{\pr{M_{\lambda, t}}_{t\in\tset}\colon \lambda\in\st{\llambda, \ulambda}}\), where for all \(\lambda\in\nnegreals\), \(\pr{M_{\lambda,t}}_{t\in\tset}\) is the Markov semigroup induced by the Poisson generator~\(\linpoissgen{\lambda}\).
Second, Eqn.~\eqref{eqn:poissgr:poissgen is derivative} implies the condition in Eqn.~\eqref{eqn:uniformly bounded rate}.
Third, it turns out that \(\poissgr_t\) is downward continuous for all \(t\in\posreals\), which makes the Nisio semigroup~\(\pr{\poissgr_t}_{t\in\tset}\) a sublinear Markov semigroup with uniformly bounded rate; the proof of this result exclusively uses results from \citep{2019ErreygersDeBock,2019ErreygersDeBock-arXiv}, and can be found in
\ifarxiv
  \appendixref{asec:proofs}.
\fi
\begin{proposition}
\label{prop:poisson group}
  The family~\(\pr{\poissgr_t}_{t\in\tset}\) is a sublinear Markov semigroup that has uniformly bounded rate.
\end{proposition}

\section{Conclusion}
We have investigated sublinear expectations for countable-state uncertain processes, more specifically (i) how an upper (or sublinear) expectation on the (bounded) finitary functions is in one-to-one correspondence with a consistent collection of finite-dimensional upper expectations; (ii) how, under the condition of downward continuity, we can extend such an upper expectation to a sublinear expectation on the bounded above or below measurable functions that is upward and downward continuous on (a large part of) its domain; and (iii) how we can get an upper expectation on the bounded finitary functions that satisfies the downward-continuity condition by combining a downward-continuous upper expectation with a sublinear Markov semigroup.

Inevitably, we have left some questions unanswered.
The first one that springs to mind is whether in \corollaryref{cor:extension is unique} we can achieve uniqueness on the entire domain~\(\genmeas\pr{\gendomain}\) instead of only on the bounded below part~\(\genbbmeas\pr{\gendomain}\).
A second interesting question is whether the condition in \theoremref{the:alt sufficient condition for downward continuity} is not only sufficient but also necessary; the results in \citep{2022Erreygers} suggest that this is not the case, and may provide guidance to find a necessary condition in terms of the upper expectation.
These results could also provide the key to relax the requirement of sublinearity in \theoremref{the:alt sufficient condition for downward continuity} to convexity.
Finally, in \sectionref{ssec:examples} we did not investigate whether every (downward-continuous) bounded upper rate operator generates a semigroup of (downward-continuous) upper transition operators.
I have confirmed--but not yet published---that this is indeed the case for bounded upper rate operators; dealing with unbounded upper rate operators---which can only occur if \(\stsp\) is infinite---appears to be more involved though, and is the subject of ongoing research.

Even more general would be to shift to sublinear Feller processes, so Markov processes with uncountable state spaces.
\Citet{2021NendelRockner} have already taken the first steps in the investigation of sublinear Feller semigroups, but a lot more work is needed to establish a robust version---so with sublinear expectations instead of linear ones---of the Dynkin--Kinney--Blumenthal Theorem \citep[Theorem~II.7.17]{1994Rogers-Diffusions}.

\appendix




\acks{%
This work is supported by the Research Foundation---Flanders (FWO) (project number 3G028919).
I would like to thank Jasper De Bock for his everpresent enthusiasm to discuss (the) mathematics (in this work), as well as the participants of the ImPRooF workshop---where I presented a preliminary verison of this work---for the stimulating discussions.
I am also grateful to the four anonymous reviewers of this work; their constructive comments and super suggestions have surely made a difference.
}



\renewcommand{\deDe}{de}
\bibliography{erreygers23}

\renewcommand{\deDe}{De}
\ifappend
\ifarxiv
\section{Relegated proofs}
\label{asec:proofs}
\else\clearpage
\fi
\begin{proof}\textbf{of~\theoremref{the:convex extension}.}
  Let \(\domprevs_{\cprev}\coloneqq\st{\prev\in\prevs_{\gendomain}\colon \cprev^*\pr{\prev}<+\infty}\).
  Since \(\cprev\) is downward continuous, we know from \lemmaref{lem:downward for convex} that every linear expectation~\(\prev\in\domprevs_{\cprev}\) is downward continuous.
  Consequently, if follows from the Daniell--Stone Theorem that for all \(\prev\in\domprevs_{\cprev}\), \(\prev=\hat\prev\vert_{\gendomain}\) with
  \begin{equation*}
    \hat{\prev}
    \colon \genmeas\to\extreals
    \colon g\mapsto \int g\mathrm{d}\prob_{\prev}.
  \end{equation*}
  It follows immediately from this and \lemmaref{lem:representation} that \(\cprevext\) is well defined and extends \(\cprev\).

  On several occasions, we will need that for all \(f\in\genmeas\pr{\gendomain}\) and \(\prev\in\domprevs_{\cprev}\), \(\cprev^*\pr{\prev}\in\reals\) (due to \lemmaref{lem:representation}) and
  \begin{equation}
  \label{eqn:proof of convex extension:upper bound on hatprev}
    \hat{\prev}\pr{f}
    \leq \cprevext\pr{f} + \cprev^*\pr{\prev}.
  \end{equation}

  Next, we show that \(\cprevext\) is a convex expectation.
  The extension~\(\cprevext\) is a nonlinear expectation: (i) \(\genmeas\pr{\gendomain}\) includes all constant real functions because \(\gendomain\subseteq\genmeas\pr{\gendomain}\) and \(\gendomain\) includes all constant real functions; (ii) \(\cprevext\) is isotone because the Lebesgue integral is isotone on~\(\genmeas\pr{\gendomain}\) \citep[Chapter~8, Theorem~5~(iv)]{1997Fristedt-Modern}; and (iii) \(\cprevext\) is constant preserving because it extends~\(\cprev\) and \(\cprev\) is constant preserving.
  To verify that \(\cprevext\) is convex, we fix some \(f,g\in\genmeas\pr{\gendomain}\) and \(\lambda\in\cci{0}{1}\) such that \(f+g\) is meaningful and in~\(\genmeas\pr{\gendomain}\) and \(\lambda\cprevext\pr{f}+\pr{1-\lambda}\cprevext\pr{g}\) is meaningful.
  If \(\lambda=0\) or \(\lambda=1\), clearly \(\cprevext\pr{\lambda f+\pr{1-\lambda} f}=\lambda\cprevext\pr{f}+\pr{1-\lambda}\cprevext\pr{g}\); hence, without loss of generality we may assume that \(0<\lambda<1\).
  Due to symmetry, and because \(\lambda\cprevext\pr{f}+\pr{1-\lambda}\cprevext\pr{g}\) is meaningful, we need to distinguish three cases: (i) \(\cprevext\pr{f}=+\infty\) and \(\cprevext\pr{g}>-\infty\); (ii) \(\cprevext\pr{f}\) and \(\cprevext\pr{g}\) both real; and (iii) \(\cprevext\pr{f}=-\infty\) and \(\cprevext\pr{g}<+\infty\).
  In the first case, the required inequality holds trivially.
  In the second case, it follows from Eqn.~\eqref{eqn:proof of convex extension:upper bound on hatprev} that for all \(\prev\in\domprevs_{\cprev}\), \(\hat\prev\pr{f}<+\infty\) and \(\hat\prev\pr{g}<+\infty\), so \(\lambda\hat\prev\pr{f}+\pr{1-\lambda}\hat\prev\pr{g}\) is meaningful and, due to the linearity of~\(\hat\prev\) \citep[Chapter~8, Theorem~5~(i)]{1997Fristedt-Modern}, equal to~\(\hat\prev\pr{\lambda f+\pr{1-\lambda}g}\).
  Similarly, in the third case, it follows from Eqn.~\eqref{eqn:proof of convex extension:upper bound on hatprev} that for all \(\prev\in\domprevs_{\cprev}\), \(\hat\prev\pr{f}=-\infty\) and \(\hat\prev\pr{g}<+\infty\), so \(\lambda\hat\prev\pr{f}+\pr{1-\lambda}\hat\prev\pr{g}\) is meaningful and, due to the linearity of~\(\hat\prev\), equal to~\(\hat\prev\pr{\lambda f+\pr{1-\lambda}g}\).
  Consequently, in the last two cases,
  \begin{align*}
    \MoveEqLeft\cprevext\pr{\lambda f+\pr{1-\lambda}g}\\
    &= \sup\st[\big]{\hat\prev\pr{\lambda f+\pr{1-\lambda}g}-\cprev^*\pr{\prev}\colon\prev\in\domprevs_{\cprev}} \\
    &= \sup\st[\big]{\hat\prev\pr{\lambda f}+\pr{1-\lambda}\hat\prev\pr{g}-\cprev^*\pr{\prev}\colon\prev\in\domprevs_{\cprev}} \\
    &\leq \lambda \sup\st[\big]{\hat\prev\pr{f}-\cprev^*\pr{\prev}\colon\prev\in\domprevs_{\cprev}}\\&\qquad+\pr{1-\lambda}\sup\st[\big]{\hat\prev\pr{g}-\cprev^*\pr{\prev}\colon\prev\in\domprevs_{\cprev}} \\
    &= \lambda\cprevext\pr{f}+\pr{1-\lambda}\cprevext\pr{g},
  \end{align*}
  as required.

  \Citet[Theorem~3.10]{2018DenkKupperNendel} show that the restriction of~\(\cprevext\) to~\(\genmeas\pr{\gendomain}\cap\bfns\pr{\genset}\supseteq\gendomain_{\delta, \mathrm{b}}\) is downward continuous on~\(\gendomain_{\delta, \mathrm{b}}\), so clearly \(\cprevext\) is downward continuous on~\(\gendomain_{\delta, \mathrm{b}}\) too.

  Proving the upward continuity on~\(\genbbmeas\pr{\gendomain}\) is straightforward.
  Fix any~\(\pr{\genbbmeas}^{\nats}\ni\pr{f_n}_{n\in\nats}\nearrow f\in\genbbmeas\pr{\gendomain}\).
  For all \(\prev\in\domprevs_{\cprev}\), \(\hat\prev\) is upward continuous on~\(\genbbmeas\)---due to the Monotone Convergence Theorem, see for example \citep[Theorem~12.1]{2017Schilling-Measures}---and therefore \(\lim_{n\to+\infty}\hat\prev\pr{f_n}=\sup_{n\in\nats}\hat\prev\pr{f_n}=\hat\prev\pr{f}\).
  From this and the isotonicity of~\(\cprevext\), it follows that
  \begin{align*}
    \lim_{n\to+\infty} \cprevext\pr{f_n}
    &= \sup\st[\big]{\cprevext\pr{f_n}\colon n\in\nats} \\
    &= \sup\st[\big]{\sup\st[\big]{\hat\prev\pr{f_n}-\cprev^*\pr{\prev}\colon\prev\in\domprevs_{\cprev}}\colon n\in\nats} \\
    &= \sup\st[\big]{\sup\st[\big]{\hat\prev\pr{f_n}-\cprev^*\pr{\prev}\colon n\in\nats}\colon\prev\in\domprevs_{\cprev}} \\
    &= \sup\st[\big]{\hat\prev\pr{f}-\cprev^*\pr{\prev}\colon\prev\in\domprevs_{\cprev}} \\
    &= \cprevext\pr{f},
  \end{align*}
  as required.

  To prove the second part of the statement, we assume that \(\cprev\) is an upper expectation.
  Recall from \lemmaref{lem:representation} that \(\cprev^*\pr{\prev}=0\) for all \(\prev\in\domprevs_{\cprev}\) and that \(\domprevs_{\cprev}\) is the set of dominated linear expectations (on~\(\gendomain\)).
  Hence, to see that \(\cprevext\) is positively homogeneous, it suffices to realise that for all \(\prev\in\domprevs_{\cprev}\) (i) \(\cprev^*\pr{\prev}=0\) due to \lemmaref{lem:representation}; and (ii) \(\hat\prev\) is homogeneous \citep[Chapter~8, Theorem~5~(i)]{1997Fristedt-Modern}.
  That \(\cprevext\) is subadditve follows from a similar argument as the one we used to prove that \(\cprevext\) is convex.
\end{proof}

\begin{proof}\textbf{of \corollaryref{cor:extension is unique}.}
  From Theorem~3.10 in~\cite{2018DenkKupperNendel}---or the functional version of Choquet's Capacitibility Theorem, see \citep[Proposition~2.1]{2019Bartl}---it follows that for all \(f\in\genbbmeas\pr{\gendomain}\cap\genbameas\pr{\gendomain}=\genmeas\pr{\gendomain}\cap\bfns\pr{\genset}\),
  \begin{equation}
  \label{eqn:proof of extension is unique:Choquet}
    \cprevext\pr{f}
    = \sup\st*{\lim_{n\to+\infty}\cprevext\pr{f_n}\colon \gendomain^{\nats}\ni\pr{f_n}_{n\in\nats}\searrow\leq f}.
  \end{equation}
  It remains for us to prove the equality in the statement for all \(f\in\genbbmeas\pr{\gendomain}\setminus\genbameas\pr{\gendomain}\), so let us fix any such~\(f\).
  Then \(\pr{f\wedge k}_{k\in\nats}\) is an increasing sequence in~\(\genbbmeas\pr{\gendomain}\cap\genbameas\pr{\gendomain}\) that converges pointwise to~\(f\), and therefore
  \begin{equation*}
    \cprevext\pr{f}
    = \lim_{k\to+\infty} \cprevext\pr{f\wedge k}
    = \sup\st[\big]{\cprevext\pr{f\wedge k}\colon k\in\nats}.
  \end{equation*}
  Because \(f\wedge k\in\genbbmeas\pr{\gendomain}\cap\genbameas\pr{\gendomain}\) for all \(k\in\nats\), it follows from this equality and Eqn.~\eqref{eqn:proof of extension is unique:Choquet} that
  \begin{align*}
    \cprevext\pr{f}
    &= \sup\st*{\lim_{n\to+\infty}\cprevext\pr{f_n}\colon k\in\nats, \gendomain^{\nats}\ni\pr{f_n}_{n\in\nats}\searrow\leq f\wedge k} \\
    &= \sup\st*{\lim_{n\to+\infty}\cprevext\pr{f_n}\colon \gendomain^{\nats}\ni\pr{f_n}_{n\in\nats}\searrow\leq f},
  \end{align*}
  as required.
\end{proof}

\begin{proof}\textbf{of \equationref{eqn:sigma cylevts vs sigma fdomain}.}
  Due to Lemma~8.1 (and Lemma~8.3) in \citep{2017Schilling-Measures}, \(\sigma\pr{\fdomain}\) is generated by the collection of level sets
  \begin{equation*}
    \mathcal{C}
    \coloneqq \st[\big]{\st{\pth\in\cadpths\colon f\pr{\pth}\geq\alpha}\colon f\in\fdomain, \alpha\in\reals}.
  \end{equation*}
  Hence, it follows from Eqn.~\eqref{eqn:cylevts from indicators} that every cylinder~\(\cyl\in\cylevts\) belongs to~\(\mathcal{C}\), and therefore also to~\(\sigma\pr{\fdomain}\).
  Consequently, \(\sigma\pr{\cylevts}\subseteq\sigma\pr{\fdomain}\).

  To prove that \(\sigma\pr{\fdomain}\subseteq\sigma\pr{\cylevts}\), it suffices to verify that any level set in~\(\mathcal{C}\) is a cylinder.
  To this end, we fix any \(f\in\fdomain\) and \(\alpha\in\reals\).
  By definition of~\(\fdomain\), there are some \(\ftset\in\ftsets\) and \(g\in\bfns\pr{\stsp^{\ftset}}\) such that \(f=g\circ\pi_{\ftset}\).
  Let \(A\coloneqq\st{x\in\stsp^{\ftset}\colon g\pr{x}\geq \alpha}\).
  Then clearly
  \begin{align*}
    \st{\pth\in\cadpths\colon f\pr{\pth}\geq\alpha}
    = \st{\pth\in\cadpths\colon \pi_{\ftset}\pr{\pth}\in A},
  \end{align*}
  so this level set is indeed a cylinder.
\end{proof}

\begin{proof}\textbf{of \lemmaref{lem:derived set function}.}
  That \(\altprob_{\prev}\) is finitely additive with \(\altprob_{\prev}\pr{\cadpths}=1\) follows immediately because \(\prev\) is a linear expectation.
  Hence, we focus on the second part of the statement.

  First, we assume that \(\prev\) is downward continuous.
  Then it follows immediately from the Daniell--Stone Theorem that \(\altprob_{\prev}=\prob_{\prev}\vert_{\cylevts}\), and therefore \(\altprob_{\prev}\) is countably additive.

  Second, we assume that \(\altprob_{\prev}\) is countably additive.
  Then it is well known, see for example Proposition~9 in~\citep[Chapter~7]{1997Fristedt-Modern} or Lemma~4.3 in~\citep[Chapter~II]{1994Rogers-Diffusions}, that for any decreasing~\(\pr{\cyl_n}_{n\in\nats}\in\cylevts^{\nats}\)---meaning that \(\cyl_n\supseteq \cyl_{n+1}\) for all \(n\in\nats\)---with \(\bigcap_{n\in\nats}\cyl_n=\emptyset\),
  \begin{equation}
  \label{eqn:proof of derived set function:to zero}
    \lim_{n\to+\infty} \altprob_{\prev}\pr{\cyl_n}
    = 0.
  \end{equation}
  To show that \(\prev\) is downward continuous, we fix any \(f\in\fdomain\) and any decreasing sequence~\(\pr{f_n}_{n\in\nats}\in\fdomain^{\nats}\) that converges pointwise to~\(f\).
  Then
  \begin{equation}
  \label{eqn:proof of derived set function:to reduced sequence}
    \prev\pr{f_n}-\prev\pr{f}
    = \prev\pr{f_n-f}
    \geq 0
    \quad\text{for all } n\in\nats.
  \end{equation}
  Obviously, \(\pr{f_n-f}_{n\in\nats}\) is a decreasing sequence in~\(\fdomain\) that converges pointwise to~\(0\).

  Fix any \(\epsilon\in\posreals\), and let \(\beta\coloneqq\norm{f_1-f}=\sup f_1-f\).
  Then for all \(n\in\nats\), we let \(\cyl_n \coloneqq \st{\pth\in\cadpths\colon f_n\pr{\pth}-f\pr{\pth}>\epsilon}\); it is a bit laborious to verify that \(\cyl_n\in\cylevts\), so we leave this as an exercise to the reader.
  This way, \(\pr{\cyl_n}_{n\in\nats}\) is a decreasing sequence in~\(\cylevts\) with \(\bigcap_{n\in\nats} \cyl_n=\emptyset\), and for all \(n\in\nats\), \(f_n-f\leq \epsilon+\beta \indica{\cyl_n}\) and therefore
  \begin{equation*}
    \prev\pr{f_n-f}
    \leq \epsilon + \prev\pr{\indica{\cyl_n}}
    = \epsilon + \altprob_{\prev}\pr{\cyl_n}.
  \end{equation*}
  It follows from this and Eqn.~\eqref{eqn:proof of derived set function:to zero} that
  \begin{equation*}
    \lim_{n\to+\infty} \prev\pr{f_n-f}
    \leq \lim_{n\to+\infty} \epsilon+\beta\altprob_{\prev}\pr{\cyl_n}
    = \epsilon.
  \end{equation*}
  Since this inequality holds for any strictly positive real number~\(\epsilon\), we infer from it and the one in Eqn.~\eqref{eqn:proof of derived set function:to reduced sequence} that
  \begin{equation*}
    \lim_{n\to+\infty} \prev\pr{f_n}
    = \prev\pr{f},
  \end{equation*}
  as required.
\end{proof}

\begin{proof}\textbf{of \theoremref{the:alt sufficient condition for downward continuity}.}
  To prove that \(\uprev\) is downward continuous, we recall from Proposition~\ref{prop:fidis to upper expectation} that \(\uprev\) is an upper expectation.
  By \lemmaref{lem:representation,lem:downward for convex}, it suffices to verify that every dominated linear expectation~\(\prev\) in
  \begin{equation*}
    \domprevs_{\uprev}
    \coloneqq\st{\prev\in\prevs_{\fdomain}\colon\pr{\forall f\in\fdomain}~\prev\pr{f}\leq\uprev\pr{f}}
  \end{equation*}
  is downward continuous.
  So fix any \(\prev\in\domprevs_{\uprev}\), and let
  \begin{equation*}
    \altprob_{\prev}
    \colon \cylevts\to\cci{0}{1}
    \colon \cyl\mapsto \prev\pr{\indica{\cyl}}.
  \end{equation*}
  We know from \lemmaref{lem:derived set function} that \(\altprob_{\prev}\) is finitely additive with \(\altprob_{\prev}\pr{\cadpths}=1\), and that \(\prev\) is downward continuous if and only if \(\altprob_{\prev}\) is countably additive.
  Hence, it suffices to show that \(\altprob_{\prev}\) is countably additive, and we will do so by checking that the conditions in \lemmaref{lem:prop24} are satisfied.

  First, fix any \(\ftset\in\ftsets\), and let
  \begin{equation*}
    \altprob_{\prev}^{\ftset}
    \colon \wp\pr{\stsp^{\ftset}}\to\cci{0}{1}\colon
    A\mapsto
    \altprob_{\prev}\pr[\big]{\pi_{\ftset}^{-1}\pr{A}}
    = \prev\pr[\big]{\indica{\pi_{\ftset}^{-1}\pr{A}}}.
  \end{equation*}
  Clearly, \(\altprob_{\prev}^{\ftset}\) is a non-negative set function with \(\altprob_{\prev}^{\ftset}\pr{\stsp^{\ftset}}=\altprob_{\prev}\pr{\cadpths}=1\) that is finitely additive.
  By a standard result in measure theory---see for example Proposition~9 in~\citep[Chapter~7]{1997Fristedt-Modern} or Lemma~4.3 in \citep[Chapter~II]{1994Rogers-Diffusions}---\(\altprob_{\prev}^{\ftset}\) is countably additive, and therefore a probability measure, if and only if for any decreasing sequence~\(\pr{A_k}_{k\in\nats}\) in \(\stsp^{\ftset}\) with \(\bigcap_{k\in\nats} A_k=\emptyset\), \(\lim_{k\to+\infty}\altprob_{\prev}^{\ftset}\pr{A_k}=0\).
  For any such sequence~\(\pr{A_k}_{k\in\nats}\), the corresponding sequence of indicators~\(\pr{\indica{\pi_{\ftset}^{-1}\pr{A_k}}}_{k\in\nats}\in\fdomain^{\nats}\) clearly decreases to~\(0\), and therefore
  \begin{multline*}
    0
    \leq \lim_{k\to+\infty} \altprob_{\prev}^{\ftset}\pr{A_k}
    \leq \lim_{k\to+\infty} \uprev\pr{\indica{\pi_{\ftset}^{-1}\pr{A_k}}} \\
    = \lim_{k\to+\infty} \uprev_{\ftset}\pr{\indica{A_k}}
    = 0,
  \end{multline*}
  where for the final equality we used that \(\uprev_{\ftset}\) is downward continuous and constant preserving.

  Next, fix some \(n\in\nats\) and \(t\in\cci{0}{n}\).
  Then for all \(s\in\tset\setminus\st{t}\),
  \begin{equation*}
    \altprob_{\prev}^{\st{t,s}}\pr{D_{\st{t,s}}^{\neq}}
    \leq \uprev_{\st{s,t}}\pr{d_{\st{t,s}}^{\neq}}.
  \end{equation*}
  Hence,
  \begin{align*}
    \limsup_{s\to t} \frac{\altprob_{\prev}^{\st{t,s}}\pr{D_{\st{t,s}}^{\neq}}}{\abs{s-t}}
    \leq \limsup_{s\to t} \frac{\uprev_{\st{t,s}}\pr{d_{\st{t,s}}^{\neq}}}{\abs{s-t}}
    \leq \lambda_n,
  \end{align*}
  as required.
\end{proof}

\begin{proof}\textbf{of Proposition~\ref{prop:poisson group}.}
  We have already established that \(\pr{\poissgr_t}_{t\in\tset}\) is a semigroup of upper transition operators, so it remains for us to verify (i) that \(\poissgr_t\) is downward continuous for all \(t\in\posreals\), and (ii) that \(\pr{\poissgr_t}_{t\in\tset}\) has uniformly bounded rate.

  To verify that \(\poissgr_t\) is downward continuous for all \(t\in\posreals\), we fix some \(t\in\posreals\) and \(z\in\nnegints\), and consider any \(\bfnsstsp^{\nats}\ni\pr{f_n}_{n\in\nats}\searrow f\in\bfnsstsp\).
  On the one hand, since \(\poissgr_t\) is isotone, \(\pr{\br{\poissgr_t f_n}\pr{z}}_{n\in\nats}\) decreases, with \(\lim_{n\to+\infty} \br{\poissgr_t f_n}\pr{z}\geq\br{\poissgr_t f}\pr{z}\).
  On the other hand, for all \(n\in\nats\), it follows from the subadditivity of~\(\poissgr_t\) that
  \begin{equation*}
    \br{\poissgr_t f_n}\pr{z}
    \leq \br{\poissgr_t \pr{f_n-f}}\pr{z}+\br{\poissgr_t f}\pr{z}.
  \end{equation*}
  Hence, it suffices for us to show that
  \begin{equation}
  \label{eqn:proof of lemma:limit}
    \lim_{n\to+\infty} \br{\poissgr_t \pr{f_n-f}}\pr{z}
    \leq 0.
  \end{equation}
  For all \(n\in\nats\), let
  \begin{equation*}
    \tilde{f}_n
    \colon \nnegints\to\reals
    \colon x\mapsto \max\st[\big]{f_n\pr{y}-f\pr{y}\colon y\in\nnegints, y\leq x}.
  \end{equation*}
  It is easy to verify that for all \(n\in\nats\), \(\tilde{f}_n\) is a bounded function that dominates~\(f_n-f\), so it  follows from the isotonicity of~\(\poissgr_t\) that
  \begin{equation*}
    \br{\poissgr_t \pr{f_n-f}}\pr{z}
    \leq \br{\poissgr_t \tilde{f}_n}\pr{z}.
  \end{equation*}
  Moreover, since \(\tilde{f}_n\) is increasing (in the sense that \(\tilde{f}_n\pr{z}\leq\tilde{f}_n\pr{y}\) whenever \(z\leq y\)), it follows from Theorem~15, Proposition~16 and Eqn.~(18) in \citep{2019ErreygersDeBock} that
  \begin{equation*}
    \br{\poissgr_t \pr{f_n-f}}\pr{z}
    \leq \sum_{y=z}^{+\infty} \tilde{f}_n\pr{y}\psi_{\ulambda t}\pr{\st{y-z}}
    = \int \tilde{f}_n\pr{z+\noarg}\mathrm{d}\psi_{\ulambda t},
  \end{equation*}
  where \(\psi_{\ulambda t}\colon\wp\pr{\nnegints}\to\cci{0}{1}\) is the probability measure corresponding to the Poisson distribution with parameter~\(\ulambda t\).
  Finally, it is easy to verify that \(\pr{\tilde{f}_n}_{n\in\nats}\) is monotone and decreases pointwise to~\(0\), so a straightforward application of the Monotone Convergence Theorem yields
  \begin{equation*}
    \lim_{n\to+\infty} \int \tilde{f}_n\pr{z+\noarg}\mathrm{d}\psi_{\ulambda t}
    = 0.
  \end{equation*}
  Eqn.~\eqref{eqn:proof of lemma:limit} follows from this equality and the previous inequality, and this finalises our proof for the downward continuity.

  Finally, we verify that the sublinear Markov semigroup~\(\pr{\poissgr_t}_{t\in\tset}\) has uniformly bounded rate---so satisfies Eqn.~\eqref{eqn:uniformly bounded rate}.
  First, note that due to constant additivity,
  \begin{multline*}
    \limsup_{t\searrow0}\frac1{t}\sup\st*{\br{\poissgr_t \pr{1-\indica{x}}}\pr{x}\colon x\in\stsp} \\
    = \limsup_{t\searrow0}\sup\st*{\frac{\br{\poissgr_t \pr{-\indica{x}}}\pr{x}-\pr{-\indica{x}\pr{x}}}{t}\colon x\in\stsp}.
  \end{multline*}
  It follows from this, the definition of the norms~\(\norm{\noarg}\) and \(\zopnorm{\noarg}\) and Eqn.~\eqref{eqn:poissgr:poissgen is derivative} that
  \begin{align*}
    \MoveEqLeft\limsup_{t\searrow0}\frac1{t}\sup\st*{\br{\poissgr_t \pr{1-\indica{x}}}\pr{x}\colon x\in\stsp} \\
    &\leq\limsup_{t\searrow0} \st*{\norm*{\frac{\poissgr\pr{-\indica{x}}-\eye\pr{-\indica{x}}}{t}}\colon x\in\stsp} \\
    &\leq \lim_{t\searrow0}\zopnorm*{\frac{\poissgr_t-\eye}{t}} = \zopnorm{\poissgen}
    <+\infty,
  \end{align*}
  where the strict inequality holds because \(\poissgen\) is a bounded operator.
\end{proof}
\fi

\end{document}